\documentclass[a4paper]{amsart}

\usepackage[applemac]{inputenc}
\usepackage[T1]{fontenc}
\usepackage[english]{babel}
\usepackage{babelbib}
\usepackage{enumerate}
\usepackage{url}
\usepackage{amssymb}
\usepackage{hyperref}
\usepackage{color}
\usepackage{tkz-fct}
\usepackage{tikz}
\usepackage{tikz-cd}
\usetikzlibrary{matrix}
\usetikzlibrary{arrows}

\newcommand{\A}{\mathbb{A}}
\newcommand{\C}{\mathbb{C}}
\newcommand{\D}{\mathbb{D}}

\newcommand{\N}{\mathbb{N}}

\newcommand{\Q}{\mathbb{Q}}
\newcommand{\R}{\mathbb{R}}

\newcommand{\Z}{\mathbb{Z}}

\newcommand{\M}{\cal{M}}

\newcommand{\rig}{\textup{rig}}
\newcommand{\an}{\textup{an}}
\renewcommand{\H}{\Cohom}

\newcommand{\too}{\longrightarrow}
\renewcommand{\phi}{\varphi}
\renewcommand{\epsilon}{\varepsilon}
\renewcommand{\ker}{\Ker}

\newcommand{\cal}{\mathcal}

\renewcommand{\Im}{\parteimmaginaria}

\newcommand{\khat}{\mathcal{H}}

\DeclareMathOperator{\parteimmaginaria}{\textup{Im}}

\DeclareMathOperator{\pr}{pr}

\DeclareMathOperator{\FHom}{\mathcal{H}\textit{om}}

\DeclareMathOperator{\im}{Im}
\DeclareMathOperator{\Ker}{Ker}

\DeclareMathOperator{\coker}{Coker}

\DeclareMathOperator{\red}{red}

\DeclareMathOperator{\BB}{B}
\DeclareMathOperator{\CC}{C}
\DeclareMathOperator{\EE}{E}

\DeclareMathOperator{\ZZ}{Z}

\DeclareMathOperator{\Supp}{Supp}
\DeclareMathOperator{\Cohom}{H}
\DeclareMathOperator{\ev}{ev}
\DeclareMathOperator{\id}{id}

\DeclareMathOperator{\Max}{Max}

\DeclareMathOperator{\hotimes}{\hat\otimes}

\renewcommand{\le}{\leqslant}
\renewcommand{\ge}{\geqslant}

\newcommand{\Fc}{F}
\newcommand{\Gc}{G}

\newcommand{\Oc}{\mathcal{O}}

\newcommand\wc{{\mkern 2mu\cdot\mkern 2mu}}
\newcommand\va{|\wc|}
\newcommand\nm{\|\wc\|}

\theoremstyle{plain}
\newtheorem{theorem}{Theorem}[section]
\newtheorem{lemma}[theorem]{Lemma}
\newtheorem{proposition}[theorem]{Proposition}
\newtheorem{corollary}[theorem]{Corollary}
\newtheorem{claim}[theorem]{Claim}

\theoremstyle{definition}
\newtheorem{definition}[theorem]{Definition}

\theoremstyle{remark}
\newtheorem{remark}[theorem]{Remark}

\numberwithin{equation}{section}

\begin{document}
\title{Notions of Stein spaces in non-archimedean geometry}

\author{Marco Maculan}
\email{marco.maculan@imj-prg.fr}
\address{Institut de Math\'ematiques de Jussieu, Universit\'e Pierre et Marie Curie, 4 place Jussieu, F-75252 Paris}

\author{J\'er\^ome Poineau}
\email{jerome.poineau@unicaen.fr}
\address{Laboratoire de math\'ematiques Nicolas Oresme, Universit\'e de Caen Normandie, BP 5186, F-14032 Caen Cedex}

\begin{abstract} 
Let $k$ be a non-archimedean complete valued field and $X$ be a $k$-analytic space in the sense of Berkovich. In this note, we prove the equivalence between three properties: 1) for every complete valued extension~$k'$ of~$k$, every coherent sheaf on~$X \times_{k} k'$ is acyclic; 2) $X$ is Stein in the sense of complex geometry (holomorphically separated, holomorphically convex) and higher cohomology groups of the structure sheaf vanish (this latter hypothesis is crucial if, for instance, $X$ is compact); 3) $X$ admits a suitable exhaustion by compact analytic domains considered by Liu in his counter-example to the cohomological criterion for affinoidicity.

When~$X$ has no boundary the characterization is simpler: in~2) the vanishing of higher cohomology groups of the structure sheaf is no longer needed, so that we recover the usual notion of Stein space in complex geometry; in 3) the domains considered by Liu can be replaced by affinoid domains, which leads us back to Kiehl's definition of Stein space.
\end{abstract}

\maketitle

\section{Introduction}

\subsection{Background} A complex analytic space $X$ is said to be \emph{Stein} if it is
\begin{itemize}
\item \emph{holomorphically separated}: for all points $x, y \in X$, there is a global holomorphic function $f \colon X \to \C$ such that $f(x) \neq f(y)$;
\item \emph{holomorphically convex}: for every compact $K \subset X$, the holomorphically convex hull of~$K$:
$$ \hat{K}_X = \{ x \in X : |f(x)| \le \| f\|_K \textup{ for all } f \in \Oc(X)\},$$
where $\| f\|_K = \sup_{K} |f|$, is compact.
\end{itemize}

The so-called Theorem B of Cartan (\cite{CartanThmABStatement}, \cite{CartanThmABProof}) states that for a coherent sheaf $F$ on a Stein space $X$ the cohomology group $\H^q(X, F)$ vanishes for all $q \ge 1$. Conversely, a complex analytic space $X$, countable at infinity, on which the higher cohomology of every coherent sheaf vanishes is Stein. 

As for non-archimedean analysis, Stein spaces have been investigated in the framework of rigid geometry by Kiehl (\cite{KiehlStein}), L\"utkebohmert (\cite{LutkebohmertStein}) and later on by Liu (\cite{LiuCRAS}, \cite{LiuBordeaux}, \cite{LiuTohoku}). The lack of local compactness of rigid spaces makes the notion of holomorphically convex hard to handle. Instead Kiehl considers a different property called \emph{quasi-Stein} (renamed here \emph{being W-exhausted by affinoid domains}) reminiscent of exhaustion by analytic blocks in complex analysis (\cite[Chapter IV, \S 4]{GrauertRemmertSteinSpaces}).

In the present note, non-archimedean Stein spaces are studied in the context of Berkovich analytic spaces over a complete non-archimedean field $k$ possibly trivially valued.

\begin{definition}[{\cite[Definition 2.3]{KiehlStein}}] 
A $k$-analytic space $X$ is said to be \emph{W-exhausted by affinoid domains} if it admits an affinoid cover for the G-topology $\{ D_i \}_{i \in \N}$ such that, for $i \ge 0$, $D_i$ is contained in $D_{i+1}$ and the restriction map $$\Oc_X(D_{i+1}) \too \Oc_X(D_i),$$ has dense image.
\end{definition}

\begin{remark} \
\begin{enumerate}

\item If $X$ is compact, the previous definition is equivalent to being affinoid.

\item 
The affinoid domain $D_i$ can always be supposed to be contained in the topological interior of $D_{i + 1}$. If $X$ is without boundary the topological interior of $D_{i}$ in $X$ coincides with the  interior of $D_i$ relative to $\cal{M}(k)$ in the sense of Berkovich (\emph{cf.} Proposition 2.5.8 (iii) and Corollary 2.5.13 (ii) \cite{Berkovich90}). We then recover the property called \emph{Stein} by Kiehl.
 
 \end{enumerate}

\end{remark}

Kiehl proved the following version of Cartan's Theorem B:

\begin{theorem}[{\cite[Satz 2.4]{KiehlStein}}] \label{Thm:SpacesExhaustedByWeierstrassAreStein} Let $X$ be $k$-analytic space W-exhausted by affinoid domains $\{ D_i\}_{i \in \N}$ and let $F$ be a coherent sheaf on $X$. Then,
\begin{enumerate}
\item $F(X)$ is dense in $F(D_i)$ for all $i \ge 0$;
\item $\H^q(X, F) = 0$ for all $q \ge 1$.
\end{enumerate}
\end{theorem}

\begin{remark} Kiehl works with rigid spaces so that $k$ should be supposed non-trivially valued and $X$ strict. However his proof goes through \emph{verbatim} under the hypothesis of Theorem \ref{Thm:SpacesExhaustedByWeierstrassAreStein}.  One can also argue that (1) follows from the density of $\Oc(D_{i+1})$ in $\Oc(D_{i})$ (see \cite[Hilfssatz 2.5]{KiehlStein}), while the compatibility of cohomology with field extensions (Theorem \ref{thm:extensioncohomology}) allows one to deduce (2) from the original result of Kiehl.
\end{remark}

In the late eighties, Liu showed in a series of papers (\cite{LiuCRAS}, \cite{LiuBordeaux}, \cite{LiuTohoku}) that Kiehl's notion is too restrictive: in \cite{LiuCRAS}, he constructs a compact $k$-analytic space which is not affinoid but whose normalization is, and, in \cite{LiuTohoku}, he exhibits an analytic domain of the $2$-dimensional disc which is not affinoid but whose higher coherent cohomology vanishes.

To ensure that the spaces he constructs have no higher coherent cohomology, Liu proves a useful criterion, valid for compact analytic spaces. Let us recall it here.

\begin{definition} A $k$-analytic space $X$ is said to be:
\begin{itemize}
\item \emph{rig-holomorphically separable} if for all rigid points $x, x' \in X_\rig$ there exists a holomorphic function $f \in \Oc(X)$ such that $f(x) = 0$ and $f(x') = 1$;
\item \emph{cohomologically Stein} if for every coherent sheaf $F$ of $\Oc_X$-modules and every $q \ge 1$ the cohomology group $\H^q(X, F)$ vanishes.
\end{itemize}
\end{definition}

\begin{theorem}[{\cite[Th\'eor\`eme 2]{LiuTohoku}}] \label{Thm:LiuTohoku} Suppose $k$ non-trivially valued. Let $X$ be a separated \emph{compact} strictly $k$-analytic space. Then the following conditions are equivalent:
\begin{enumerate}
\item $X$ is rig-holomorphically separable and $\Oc_X$ is acyclic;
\item $X$ is cohomologically Stein.
\end{enumerate}
Furthermore, if $X$ satisfies one of the preceding equivalent conditions, then it can be embedded as an analytic domain in a strictly $k$-affinoid space. In particular, if $X$ is cohomologically Stein, then $X_{k'}$ is cohomologically Stein for every analytic extension $k'$ of $k$.
\end{theorem}

\subsection{Statement of the results} In order to state the results, let us introduce the following definitions:

\begin{definition} Let $X$ be a $k$-analytic space.
\begin{itemize}
\item The \emph{holomorphic convex hull} of a compact subset $K$ of $X$ is
$$ \hat{K}_X := \{ x \in X : |f(x)| \le \| f\|_K \textup{ for all } f \in \Oc(X) \},$$
where, for $f \in \Oc(X)$,  $ \| f \|_K := \sup_{x \in K} |f(x)|$. 
\end{itemize}
The $k$-analytic space $X$ is said to be:
\begin{itemize}
\item \emph{holomorphically separable} if for all points $x, x' \in X$ there exists a holomorphic function $f \in \Oc(X)$ such that $|f(x)| \neq |f(x')|$;
\item \emph{holomorphically convex} if for every compact subset $K \subset X$ the holomorphically convex hull~$\hat{K}_X$ of~$K$ is compact.
\end{itemize}
\end{definition}

\begin{definition}
A coherent sheaf~$F$ on a $k$-analytic space~$X$ is said to be \emph{universally acyclic} if, for every analytic extension~$k'$ of~$k$, $F_{k'}$ is acyclic on~$X_{k'}$.
\end{definition} 

Note that, by Theorem~\ref{thm:extensioncohomology}, if $k$~is non-trivially valued and $X$~is separated and countable at infinity, then any acyclic coherent sheaf on~$X$ is universally acyclic.

\begin{definition} A  $k$-analytic space $X$ which is separated, holomorphically separable, compact and on which~$\Oc_X$ is universally acyclic is called a \emph{Liu space}. 
\end{definition}

\begin{definition}
A $k$-analytic space $X$ is said to be \emph{W-exhausted by Liu domains} if it admits a cover for the G-topology $\{ D_i \}_{i \in \N}$ by Liu spaces such that, for $i \ge 0$, $D_i$ is contained in $D_{i+1}$ and the restriction map  $\Oc_X(D_{i+1}) \to \Oc_X(D_i)$ has dense image.
\end{definition}

The main results of this note are:

\begin{theorem}[{\emph{cf.} Theorem
\ref{Thm:EquivalenceBoundary}}] \label{Thm:EquivalenceBoundaryIntro} Let $X$ be a separated $k$-analytic space countable at infinity. The following are equivalent:
\begin{enumerate}
\item for every analytic extension $k'$ of $k$, $X_{k'}$ is cohomologically Stein;
\item $X$ is holomorphically convex, holomorphically separable and $\Oc_X$ is acyclic;
\item $X$ is W-exhausted by Liu domains.
\end{enumerate}
Moreover, if the valuation of $k$ is non-trivial and $X$ is strict, one may replace holomorphically separable by rig-holomorphically separable in~(2). 
\end{theorem}

When the spaces in question are without boundary, the previous characterization takes the following form:

\begin{theorem}[{\emph{cf.} Theorem \ref{Thm:EquivalenceWithoutBoundary}}] \label{Thm:EquivalenceWithoutBoundaryIntro} Let $X$ be a $k$-analytic space without boundary and countable at infinity. The following are equivalent:

\begin{enumerate}
\item for every analytic extension $k'$ of $k$, $X_{k'}$ is cohomologically Stein;
\item $X$ is holomorphically convex and holomorphically separable;
\item $X$ is W-exhausted by affinoid domains.
\end{enumerate} 
Moreover, if the valuation of $k$ is non-trivial, one may replace holomorphically separable by rig-holomorphically separable in~(2).
\end{theorem}

\begin{remark}The question whether the equivalent conditions in Theorem \ref{Thm:EquivalenceWithoutBoundaryIntro} are in turn equivalent to $X$ being cohomologically Stein remains open.
\end{remark}

The proof of Theorem \ref{Thm:EquivalenceBoundaryIntro} has two ingredients. One is Cartan's original argument to exhaust Stein spaces (\emph{cf.} \cite[Lemme p. 8-9]{Cartan}, \cite[Chapter IV, \S 3 Theorems 6-7]{GrauertRemmertSteinSpaces}). 
The other one is the compatibility of the construction of the holomorphically convex hull with extension of scalars (\emph{cf.} Proposition \ref{Prop:CompatibilityHolConvexhullsExtScalars}).

Let us state some formal consequences of Theorems \ref{Thm:EquivalenceBoundaryIntro} and \ref{Thm:EquivalenceWithoutBoundaryIntro}.

\begin{corollary} Let $X$ be a $k$-analytic space without boundary. Then $X$ is W-exhausted by Liu domains if and only if it is W-exhausted by affinoid domains.
\end{corollary}

\begin{corollary} Let $X$ be a $k$-analytic space and $k'$ an analytic extension of $k$. Then $X$ is W-exhausted by Liu domains if and only if $X_{k'}$ is.
\end{corollary}
\begin{proof}
It follows from Theorem \ref{thm:extensioncohomology}.
\end{proof}

\begin{corollary} \label{cor:FiniteCoverStein} Let $f \colon Y \to X$ be a finite morphism between separated $k$-analytic spaces. Then,
\begin{enumerate}
\item if $X$ is W-exhausted by Liu domains, then so is $Y$;
\item if $f$ is surjective and  $Y$ is W-exhausted by Liu domains, then $X$ is W-exhausted by Liu domains.
\end{enumerate}
\end{corollary}

\begin{proof} (1) By Theorem \ref{thm:extensioncohomology}, it suffices to show that, for every analytic extension~$k'$ of~$k$, $Y_{k'}$~is cohomologically Stein. This follows from the Proper Mapping Theorem of Kiehl (\emph{cf.} \cite[Proposition 3.3.5]{Berkovich90}) and vanishing of higher direct image (see also Proposition \ref{lem:FiniteOverSteinAreStein}).

(2) As for (1), it suffices to show that the higher coherent cohomology of $X$ vanishes on every analytic extension of $k$. This follows from Proposition~\ref{prop:CohSteinFiniteCovering}. 
\end{proof}

\begin{corollary} Let $X$ be a separated $k$-analytic space countable at infinity. Then,
\begin{enumerate}
\item $X$ is W-exhausted by Liu domains if and only if its reduction $X_{\red}$ is;
\item $X$ is W-exhausted by Liu domains if and only if every irreducible component of $X$ is.
 \end{enumerate}
\end{corollary}

\begin{proof} This is a direct consequence of Corollary \ref{cor:FiniteCoverStein}. Indeed, if $\textup{Irr}(X)$ denotes the set of irreducible components of $X$, then the map $\bigsqcup_{Y \in \textup{Irr}(X)} Y \to X$ is finite and surjective. The same is true for the reduction map $X_{\red} \to X$.
\end{proof}

\subsubsection*{Structure of the paper} In Section~\ref{sec:CohomologicallyStein}, we prove some properties of cohomologically Stein spaces, including a version of Cartan's Theorem~A (\emph{cf.} Proposition~\ref{prop:ThmA}), and show that they are stable finite morphisms (\emph{cf.} Proposition~\ref{lem:FiniteOverSteinAreStein}), which permits us to deduce Corollary \ref{cor:FiniteCoverStein} from Theorem \ref{Thm:EquivalenceWithoutBoundaryIntro}. In Section~\ref{sec:Liu}, we recall some results on Liu spaces and extend them to the non-strict setting. In Section~\ref{sec:HolomorphicConvexity} we prove compatibility of holomorphic convex hulls to extension of scalars (\emph{cf.} Proposition \ref{Prop:CompatibilityHolConvexhullsExtScalars}) and products (\emph{cf.} Proposition \ref{Prop:CompatibilityHolConvexhullsProducts}). In Section~\ref{sec:boundary}, we give a proof of our main Theorem~\ref{Thm:EquivalenceBoundaryIntro}. In Section~\ref{sec:ProofOfMainTheorem}, we focus on spaces with no boundary: we prove a factorisation theorem for proper morphisms on holomorphically convex spaces (\emph{cf.} Theorem~\ref{thm:Remmert}) and deduce Theorem~\ref{Thm:EquivalenceWithoutBoundaryIntro}. Appendix \ref{sec:Banachoid} recalls some results on normed structures on the space of global sections of a coherent sheaf on a $k$-analytic space (proved by A.~Pulita and the second named author in~\cite{PoineauPulitaBanachoid}). Appendix~\ref{sec:Zariski-trivial} contains a description of the Zariski-trivial analytic spaces due to A.~Ducros.

\subsubsection*{Conventions} Let $k$ be a field complete with respect to a non-archimedean valuation. An analytic extension of~$k$ is a complete valued field~$k'$ containing~$k$ whose absolute value restricts to that of~$k$. 

Let $X$ be a $k$-analytic space in the sense of V.~Berkovich (\cite[\S 1.2]{BerkovichIHES}). For each point $x\in X$, we denote by~$\khat(x)$ the associated complete residue field. A point $x \in X$ is said to be rigid if $\khat(x)$ is a finite extension of $k$: if this is the case, its (already complete) residue field is written $k(x)$. The set of rigid point is denoted by $X_{\rig}$. The spectrum of a Banach $k$-algebra $A$ is denoted by $\M(A)$.

\subsubsection*{Acknowledgements} We warmly thank P. Dingoyan, A. Ducros and J. Nicaise for their interest on this work and several useful discussions. The second named author was partially supported by the ANR project ``GLOBES'': ANR-12-JS01-0007-01 and ERC Starting Grant ``TOSSIBERG'': 637027.

\section{Cohomological vanishing} \label{sec:CohomologicallyStein}

In this section some formal properties of cohomologically Stein spaces are discussed. Proofs are closer to their scheme-theoretic analogues rather than the complex analytic ones (\emph{cf.} \cite[Proposition 73.1]{KaupKaup}) because in the complex case one usually works with the definition of Stein as holomorphically separable and holomorphic convex.

\begin{proposition}[Cartan's Theorem A] \label{prop:ThmA} Assume that $k$ is \emph{non-trivially} valued. Let $X$ be a cohomologically Stein $k$-analytic space. For each coherent sheaf~$F$ on~$X$ and each point~$x$ in~$X$, the image of the set of global sections~$F(X)$ generates the stalk~$F_{x}$ as an $\Oc_{X,x}$-module.
\end{proposition}

\begin{proof} Let~$F$ be a coherent sheaf on~$X$ and $x$~be a point of~$X$. Let us begin with:

\begin{remark} \label{Rmk:ProofThmA}
Suppose $x$ belongs to a closed $k$-analytic subspace $Z$ of $X$ on which the statement holds, that is, suppose~$F(Z)$ generates~$F_{\vert Z,x}$ as an $\Oc_{Z,x}$-module. The vanishing of the first cohomology group of the coherent sheaf of ideals on $X$ defining~$Z$ ensures the surjectivity of the restriction map $F(X) \to F(Z)$. It follows that $F(X)$ generates~$F_{\vert Z,x}$ as an $\Oc_{Z,x}$-module. By Nakayama's Lemma, it also generates~$F_{x}$ an an~$\Oc_{X,x}$-module.
\end{remark}

According to the previous remark, it suffices to prove the statement for the irreducible components of $X$. Therefore $X$ may be assumed irreducible and we can argue by induction on the dimension $d$ of $X$. If $d = 0$ there is nothing to prove. Suppose $d \ge 1$ and that the statement is true for irreducible spaces of dimension strictly smaller than~$d$.

If the Zariski-closure of the point $x$ in $X$ is not the whole space, it is an irreducible subspace of dimension smaller than $d$. Then the statement is true by induction hypothesis and Remark \ref{Rmk:ProofThmA}. Suppose from now on that $x$ is Zariski-dense in $X$. Then two cases need to be distinguished. 

First, if $X$ has a non-trivial irreducible closed analytic subset $Z$, then the dimension of $Z$ is smaller than $d$ and the statement holds for $Z$. It follows from Remark~\ref{Rmk:ProofThmA} that there exists a homomorphism $f \colon \Oc_X^n \to F$ is surjective at some point of $X$. Since the locus where $f$ is surjective is a Zariski-open subset $U$ of $X$, the point $x$ belongs to $U$ and $f$ is surjective at $x$. 

Second, if, on the contrary, there is no non-trivial closed analytic subspace of~$X$, then Proposition~\ref{Prop:ZariskiTrivialSpaces} states that $X = \M(A)$ for a local finite algebra $A$ over an analytic extension $K$ of $k$. Since is $X$ is made of one point, the result holds trivially.
\end{proof}

\begin{definition} Let $X$ be a $k$-analytic space and $M$ a finitely generated $\Oc(X)$-module. Let $P_M$ be the presheaf for the $G$-topology of $X$ associating to an analytic domain~$D$ of~$X$ the $\Oc(D)$-module $M \otimes_{\Oc(X)} \Oc(D)$. We denote by $\tilde{M}$ the sheaf of $\Oc_{X}$-modules for the $G$-topology of $X$ associated to $P_M$.
\end{definition}

\begin{lemma} Let $X$ be a $k$-analytic space and $M$ be a finitely generated $\Oc(X)$-module. Then,
\begin{enumerate}
\item for each affinoid domain~$D$ of~$X$, we have $\tilde{M}(D) = M \otimes_{\Oc(X)} \Oc(D)$;
\item the sheaf $\tilde{M}$ is coherent;
\item $\tilde{M}(D) = M \otimes_{\Oc(X)} \Oc(D)$ for every affinoid domain $D \subset X$;
\item  if $N$ is a finitely generated $\Oc(X)$-module and $\phi \colon M \to N$ is a surjective homomorphism, then the associated homomorphism of coherent sheaves $\tilde{\phi} \colon \tilde{M} \to \tilde{N}$ is surjective.
\end{enumerate}
\end{lemma}

\begin{proof} (1) is a consequence of Tate's acyclicity theorem \cite[Proposition~2.2.5]{Berkovich90}.

(2) follows from (1). 

(3) follows from (1) and the fact that the tensor product is right-exact.
\end{proof}

\begin{lemma}[{\cite[Lemme 1]{LiuTohoku}}] \label{lem:FiniteGenerationOnCompactSpaces} Let $X$ be a compact cohomologically Stein space, $F$ a coherent sheaf on $X$, $M$ an $\Oc_X(X)$-module and $\phi \colon M \to F(X)$ a homomorphism. 

Suppose that there is a finite affinoid G-cover $\{ X_i \}_{i \in I}$ of $X$ such that, for all $i \in I$, the homomorphism $\phi_{i} \colon M \otimes_{\Oc_X(X)} \Oc_X(X_i) \to F(X_i)$ is surjective.   Then, there is a finitely generated submodule $M_0$ of $M$ such that $\phi(M_0) = F(X)$.
\end{lemma}

\begin{proof} 
Let $i\in I$. Since $F(X_i)$ is a finitely generated $\Oc(X_i)$-module, there exists a finite subset~$N_{0}$ of~$M$ such that the set $\{\varphi_{i}(m\otimes 1)\}_{m\in N_{0}}$ generates $F(X_{i})$. Up to enlarging~$N_{0}$, we may assume that, for each $j\in I$, $\{\varphi_{j}(m\otimes 1)\}_{m\in N_{0}}$ generates $F(X_{j})$. The associated homomorphism of coherent sheaves $\Oc_{X}^{\# N_{0}} \to F$ is then surjective. Since $X$ is cohomologically Stein, the homomorphism induced on the global sections $\Oc_{X}(X)^{\# N_{0}} \to F(X)$ is surjective too. It follows that the result holds for the submodule~$M_{0}$ of~$M$ generated by~$N_{0}$.
\end{proof}

\begin{proposition}[{\cite[Proposition 1.2]{LiuTohoku}}] \label{Prop:GlobalSectionsCoherentSheafOnLiuSpaces} Assume that~$k$ is not trivially valued. Let $X$ be a cohomologically Stein $k$-analytic space. Then:
\begin{enumerate}
\item for each affinoid domain $D$ of $X$, the restriction $\Oc_X(X) \to \Oc_X(D)$ is a flat homomorphism;
\item for each finitely generated (resp. presented) $\Oc_X(X)$-module $M$ the canonical homomorphism $M \to \tilde{M}(X)$ is surjective (resp. bijective).
\end{enumerate}
Moreover, if $X$ is compact, then:
\begin{enumerate}
\setcounter{enumi}{2}
\item the ring $\Oc_X(X)$ is noetherian;
\item given a coherent sheaf $F$ on $X$, the $\Oc_X(X)$-module $M := F(X)$ is finitely generated and the canonical homomorphism $\tilde{M}\to F$ an isomorphism.
\end{enumerate}
\end{proposition}

\begin{proof} 
(1) Let~$D$ be an affinoid domain of~$X$. It suffices to show that for every finitely generated ideal~$I$ of~$\Oc_X(X)$ the natural homomorphism  $ I \otimes_{\Oc_X(X)} \Oc_X(D) \to I \Oc_X(D)$ is injective. Let $\tilde{I}$ be the coherent sheaf associated to $I$ and $J$ be the coherent sheaf of ideals generated by~$I$. 

Consider the natural map $\theta \colon \tilde{I} \to J$. It is surjective by definition. By Lemma~\ref{lem:FiniteGenerationOnCompactSpaces}, the map $\psi\colon I \to \tilde I(X)$ is surjective. Since $\theta(X) \circ \psi \colon I \to J(X)$ is injective, it follows that~$\theta(X)$ is injective too. Consider the coherent sheaf $F = \ker \phi$. We have just proved that $F(X)=0$ and Proposition~\ref{prop:ThmA} then implies that $F = 0$, that is to say $\theta$~is injective. By using Lemma~\ref{lem:FiniteGenerationOnCompactSpaces} on~$D$, we obtain $J(D) = I \Oc_{X}(D)$ and we conclude that the map  $$\tilde{I}(D) = I \otimes_{\Oc(X)} \Oc_X(D) \too J(D) = I \Oc_X(D)$$ is an isomorphism.

\medskip

(2) In the finitely generated case, the result follows from Lemma~\ref{lem:FiniteGenerationOnCompactSpaces}.

Let~$M$ be a finitely \emph{presented} $\Oc_X(X)$-module and 
$$ 0 \too K \too \Oc(X)^n \too M \too 0$$
be a presentation of $M$ with finitely generated kernel $K$. Thanks to~(1), the associated sequence of coherent sheaves
$$ 0 \too \tilde{K} \too \Oc_X^n \too \tilde{M} \too 0$$
is exact. This yields a commutative diagram
$$ \begin{tikzcd}
0 \ar[r]  & K \ar[d, "\alpha"] \ar[r] & \Oc(X)^n  \ar[d, equal] \ar[r]  & M \ar[d, "\beta"] \ar[r]  & 0 \\
0 \ar[r]  & \tilde{K}(X) \ar[r]  & \Oc(X)^n \ar[r]  & \tilde{M}(X) \ar[r]  & 0
\end{tikzcd}
$$
where the second row is exact because $X$ is cohomologically Stein. Since the central vertical row is the identity, it follows that $\alpha$ is injective and $\beta$ is surjective. Moreover, since $K$ is finitely generated, the map $\alpha$ is also surjective therefore $\beta$ is injective.

\medbreak

Henceforth suppose $X$ compact.

\medskip

(3) Let $I$ be an ideal of $\Oc_{X}(X)$ and $J$ be the sheaf of ideals generated by~$I$. Note that, for each affinoid domain~$D$ of~$X$, $I\Oc_{X}(D)$ is an ideal of~$\Oc_{X}(D)$, necessarily of finite type since $\Oc_{X}(D)$ is noetherian, so that~$J_{\vert D}$ is coherent. As a consequence, $J$ is coherent. Thanks to Lemma \ref{lem:FiniteGenerationOnCompactSpaces} there is a finitely generated ideal $I_0$ of $\Oc_X(X)$ such that $I_0$ generates $J(X)$. It follows that $I_{0} = I = J(X)$.

\medskip

(4) By Proposition~\ref{prop:ThmA}, the global sections $F(X)$ of $F$ generate the stalk of $F$ at every point $x \in X$, hence they also generate the global sections $F(D)$ over every affinoid domain $D$ of $X$. Applying Lemma \ref{lem:FiniteGenerationOnCompactSpaces}, we deduce that the $\Oc_X(X)$-module $M = F(X)$ is finitely generated. 

The natural map $\tilde{M} \to F$ is an isomorphism: it is surjective by definition; it is also injective as the homomorphism induced on global sections $\tilde{M}(X) \to F(X)$ is injective, according to (2). \end{proof}

\begin{remark}
If $X$ is a cohomologically Stein compact space over a trivially valued field, then $\Oc_{X}(X)$ is still noetherian. Indeed the proof of~(3) in the proposition below did not make use of the assumption on the valuation of~$k$.
\end{remark}

\begin{proposition} \label{lem:FiniteOverSteinAreStein}
Let $f \colon Y \to X$ be a finite morphism of $k$-analytic spaces. If $X$ is cohomologically Stein, then $Y$ is cohomologically Stein. 
\end{proposition}

Let us begin with the following:

\begin{lemma} \label{lem:ComputationCohomologyFiniteMorphism} Let $f \colon Y \to X$ be a finite morphism of $k$-analytic spaces. Then for  every coherent sheaf $F$ on $Y$ and all $q \ge 1$, 
$$ \H^q(X, f_\ast F) = \H^q(Y, F). $$
\end{lemma}

\begin{proof}[Proof of the Lemma] The statement follows by an argument of degeneration of the Leray spectral sequence as soon as all higher direct images vanish \cite[Exercise~8.1]{Hartshorne}. Let us prove $\textup{R}^qf_\ast F = 0$ for every coherent sheaf $F$ of $\Oc_Y$-modules and for every $q \ge 1$. Since this may be checked locally, $X$~may be supposed affinoid. In this case, by Kiehl's theorem \cite[Proposition 3.3.5]{Berkovich90},
$$\textup{R}^qf_\ast F = \H^q(Y, F) \otimes_{k} \Oc_X,$$
for all $q \ge 0$. Since $f$ is finite and $X$ is affinoid, $Y$ is affinoid too and the cohomology groups $\H^q(Y, F)$ vanish for all $q \ge 1$.
\end{proof}

\begin{proof}[{Proof of Proposition \ref{lem:FiniteOverSteinAreStein}}] In order to compute the cohomology of a coherent sheaf~$F$ on~$Y$, thanks to the preceding lemma, one is led back to computing the cohomology of the sheaf $f_\ast F$ on $X$. Since $f_{\ast} F$ is coherent and~$X$ is assumed cohomologically Stein, $\H^q(X, f_\ast F)$ vanishes for all $q \ge 1$.
\end{proof}

\begin{lemma}[{\cite[Hilfssatz 2.6]{KiehlStein}}] \label{lem:ProjectiveSystemCochainComplex} Let $\{ \CC_i^\bullet\}_{i \in \N}$ be a projective system of cochain complexes with transition maps 
$\phi_i^\bullet \colon \CC^\bullet_{i+1} \to \CC^\bullet_i$ and differentials $\partial^p_i \colon \CC^{p}_{i} \to \CC^{p+1}_i$. Let $q\in \Z$. Assume (at least) one of the following hypotheses:

\begin{enumerate}
\item for all $i \in \N$, the map $\phi_i^q \colon \CC^{q}_{i+1} \to \CC^{q}_{i}$ as well as the induced map on $q$-cocycles 
$$ \phi^q_i \colon \ZZ^q_{i+1} := \ker \partial^q_{i+1}  \too \ZZ^q_i : = \ker \partial^q_i,$$
are surjective;
\item for all $i \in \N$, the maps $\phi_i^p \colon \CC^{p}_{i+1} \to \CC^{p}_{i}$ are surjective for $p = q-1, q$ and the cohomology group 
$ \H^q(\CC_i^\bullet)$ vanishes.
\end{enumerate}
Then, the natural map
$$ \H^{q+1}(\varprojlim_{i \in \N} \CC_i^\bullet) \too \varprojlim_{i \in \N} \H^{q+1}(\CC_i^\bullet),$$
is an isomorphism.
\end{lemma}

\begin{proof}[Proof of Lemma \ref{lem:ProjectiveSystemCochainComplex}] For $p \in \Z$, define
$
\CC^p := \varprojlim \CC^p_i
$
and denote by $\partial^p \colon \CC^p \to \CC^{p+1}$ the induced differential map. Set also $\ZZ^p := \ker \partial^p$ and $\BB^{p+1} := \im \partial^p$. Assume (1). For all $i \in \N$, the map at the level of coboundaries,
$$ \phi^{q+1}_i \colon \BB^{q+1}_{i+1} := \im \partial^q_{i+1}  \too \BB^{q+1}_i : = \im \partial^q_i,$$
is surjective. Indeed, in the commutative diagram
$$
\begin{tikzcd}
0  \ar[r] & \ZZ^q_{i+1} \ar[d] \ar[r] & \CC^q_{i+1} \ar[r] \ar[d, "\phi_i^q"] & \BB^{q+1}_{i+1} \ar[r] \ar[d, "\phi_i^{q+1}"] & 0 \\
0 \ar[r] & \ZZ^q_i \ar[r] & \CC^q_i  \ar[r]& \BB^{q+1}_i \ar[r] & 0
\end{tikzcd}
$$
the central vertical arrow is surjective, hence so is the vertical one on the right.  As a conseuqence, the projective system of short exact sequences
$$ 0 \too \BB^{q+1}_i \too \ZZ^{q+1}_i \too \H^{q+1}(\CC_i^\bullet) \too 0, $$
satisfies the Mittlag-Leffler condition and the natural map
$$ \varprojlim_{i \in \N} \ZZ^{q+1}_i / \varprojlim_{i \in \N} \BB^{q +1 }_i \too \varprojlim_{i \in \N} \H^{q+1}(\CC_i^\bullet),$$
is an isomorphism.

The hypothesis that $\phi_i^q$ induces a surjective map on $q$-cocycles implies that  the projective system of short exact sequences
$$ 0 \too \ZZ^q_i \too \CC^q_i \too \BB^{q+1}_i \too 0, $$
satisfies the Mittlag-Leffler condition. The projective limit $\varprojlim_{i \in \N} \BB^{q+1}_i$ is therefore identified with the quotient $\varprojlim_{i \in \N} \CC^q_i / \varprojlim_{i \in \N} \ZZ^q_i$. The latter is in turn isomorphic to $\CC^q/ \ZZ^q$ as the natural map
$$ \ZZ^{q} \too \varprojlim_{i \in \N} \ZZ^{q}_i,$$
is an isomorphism (``kernels commutes with projective limits''). Summing up, $\BB^{q+1}$ is isomorphic to the projective limit $\varprojlim_{i \in \N} \BB^q_i$. This concludes the proof of the statement under assumption~(1).

Assume now (2): it suffices to prove that (1) is verified. Let $t \in \ZZ^q_i$ be a $q$-cocyle. Since the $q$-th cohomology group $\H^q(\CC_i^\bullet)$ vanishes, there exists $s \in \CC^{q-1}_i$  such that $\partial_i^{q-1}(s) = t$. By hypothesis the map $\phi_i^{q-1}$ is surjective, thus there is $s' \in \CC^{q-1}_{i+1}$ such that $\phi_i^{q-1}(s') = s$. The element $t' := \partial^{q-1}_{i+1}(s')$ is a $q$-cocycle that satisfies $\phi_i^q(t') = t$.
\end{proof}

\begin{lemma} \label{lem:ComputationCohomologyDirectLimits}  Let $X$ be a separated $k$-analytic space. For $n \in \N$, let $I_n$ be a coherent sheaf of ideals of $\Oc_X$ and $X_n$ the associated closed analytic subspace of $X$. Assume the following:
\begin{enumerate}
\item $I_{n+1} \subset I_n$ for every $n \in \N$;
\item $X_n$ cohomologically Stein for all $n \in \N$;
\item for every affinoid domain $D \subset X$, there exists an integer $r \ge 0$ such that $I_{r \rvert D}= 0$.
\end{enumerate}
Then $X$ is cohomologically Stein.
\end{lemma}

\begin{remark} The separation hypothesis here appears in order to be able to compute the cohomology of a coherent sheaf as its \v{C}ech cohomology with respect to an affinoid cover.
\end{remark}

\begin{proof} Let $F$ be a coherent sheaf of $\Oc_X$-modules and $\cal{D} = \{ D_\lambda \}_{\lambda \in \Lambda}$ an affinoid cover of $X$ for the $G$-topology. For $n \in \N$ consider the induced affinoid cover $\cal{D}_n = \{ D_\lambda \times_X X_n \}_{\lambda \in \Lambda}$ on $X_n$ and let $\CC^\bullet_n$ be the \v{C}ech complex $\CC^\bullet(\cal{D}_n, F)$. The cochain complexes $\CC^\bullet_n$ form a projective system through the maps $\CC^\bullet_{n+1} \to \CC^\bullet_n$ induced by restriction and its limit
$$ \CC^\bullet := \varprojlim_{n \in \N} \CC^\bullet_n,$$
is identified with the \v{C}ech complex $\CC^\bullet(\cal{D}, F)$. 

We show that the restriction maps $\CC_{n+1}^q \to \CC_n^q$ are surjective for all $q, n$. For $\lambda = (\lambda_0, \dots,  \lambda_q) \in \Lambda^{q + 1}$ write $D_{\lambda} = D_{\lambda_0} \times_X \cdots \times_X D_{\lambda_q}$. The $k$-analytic space $X_{n+1} \times_X D_\lambda$ is affinoid (thus cohomologically Stein) and $X_n \times_X D_\lambda$ is a closed analytic subspace of $X_{n+1} \times_X D_\lambda$. Thus the restriction map
$$ F( X_{n+1} \times_X  D_\lambda) \too F( X_n \times_X  D_\lambda),$$
is surjective. In particular the restriction map $\CC^q_{n+1} \to \CC^q_n$ is surjective.

Since the affinoid cover~$\cal{D}_n$ is acyclic and $X_n$ is cohomologically Stein, $\H^q(\CC^\bullet_n) = \H^q(X_n, F)$ vanishes for $q\ge 1$ and hypothesis~(2) of Lemma~\ref{lem:ProjectiveSystemCochainComplex} is fulfilled. On the other hand, the affinoid cover~$\cal{D}$ is acyclic, thus, for $q\ge 2$,
$$ \H^q(X, F) = \H^q(\CC^\bullet) = \varprojlim_{n \in \N} \H^q(\CC^\bullet_n) = \varprojlim_{n \in \N} \H^q(X_n, F) = 0.$$

This leaves us with proving the vanishing of $\H^1(X, F)$. For this we apply Lemma~\ref{lem:ProjectiveSystemCochainComplex} with $q = 0$. The hypothesis (1) is indeed fulfilled: for all $n \in \N$, the restriction map $\phi_{i+1}^0 \colon \CC^0_{n+1} \to \CC^0_{n}$ is surjective because $X_{n+1} \times D_\lambda$ is affinoid (hence cohomologically Stein); on the other hand, the $0$-cocyles $\ZZ^0_n$ are nothing but the global sections $F(X_n)$, therefore the surjectivity $F(X_{n+1}) \to F(X_n)$ (guaranteed by the hypothesis of $X_{n+1}$ being cohomologically Stein) reads into the surjectivity of $\phi_n^q$ at the level of $0$-cocycles.
\end{proof}

\begin{proposition} \label{prop:SteinReduction} 
Let $X$ be a separated $k$-analytic space. Then $X$ is cohomologically Stein if and only if $X_\textup{red}$ is cohomologically Stein.
\end{proposition}

\begin{proof} ($\Rightarrow$) Clear. ($\Leftarrow$) Let $N \subset \Oc_X$ the sheaf of ideals made of nilpotent functions and, for $n \in \N$, let $X_n$ the closed analytic subspace associated to $N^n$. Since affinoid algebras are noetherian, for every affinoid domain $D \subset X$, there exists an integer $r \ge 0$ such that $N^r_{\rvert D} = 0$. According to Lemma \ref{lem:ComputationCohomologyDirectLimits} it suffices to show that $X_n$ is Stein for all $n \in \N$. 

Let $n \in \N$ be a non-negative integer and $F$~be a coherent sheaf of $\Oc_{X_n}$-modules. For every integer $\ell \ge 0$, consider the short exact sequence
$$ 0 \too N^{\ell + 1} F \too N^\ell F \too N^\ell F / N^{\ell + 1} F \too 0,$$
and the associated long exact sequence of cohomology
$$ \cdots \too \H^q(X_n, N^{\ell + 1} F) \too \H^q(X_n, N^\ell F) \too \H^q(X_n, N^\ell F / N^{\ell + 1} F)  \too \cdots. $$

If $\iota \colon X_{\textup{red}} \to X_n$ is the canonical morphism, then $N^\ell F / N^{\ell + 1} F =\iota_\ast \iota^\ast N^\ell F$. Since $X_{\textup{red}}$ is supposed to be Stein, the higher cohomology of $N^\ell F / N^{\ell + 1} F$ vanishes, by Lemma~\ref{lem:ComputationCohomologyFiniteMorphism}. Employing this in the long exact sequence of cohomology groups, the homomorphism 
$$ \H^q(X_{n}, N^{\ell + 1} F) \too \H^q(X_{n}, N^\ell F),$$
is seen to be surjective for all $q \ge 1$ (actually bijective for $q \ge 2$). Since the sheaf~$N^n$ vanishes on~$X_n$, $\H^q(X_n, N^\ell F)$ vanishes for all $q \ge 1$ and all $\ell \ge 0$.
\end{proof}

\begin{proposition} \label{prop:CriterionCohSteinIrreducibleComponents} Let $X$ be a separated $k$-analytic space. Then $X$ is cohomologically Stein if and only if every irreducible component of $X$ is cohomologically Stein.
\end{proposition}

\begin{proof} ($\Rightarrow$) Clear. ($\Leftarrow$) Thanks to Lemma \ref{lem:ComputationCohomologyDirectLimits}, one may assume that $X$ has finite dimension. Indeed, let $X_n$ be the union of irreducible components of $X$ of dimension $\le n$ and let $I_n$ be the coherent sheaf of ideals defining $X_n$. Since an affinoid domain $D \subset X$ has finite dimension, the hypothesis of Lemma \ref{lem:ComputationCohomologyDirectLimits} is fulfilled. Therefore it suffices to prove the statement when $X$ is finite-dimensional.

We argue by induction on the dimension $d$ of $X$. If $d = 0$ the statement is obviously true. Suppose $d \ge 1$. If the support of a coherent sheaf $F$ on $X$ has dimension $\le d - 1$ then its higher cohomology vanishes (whether $\Supp F$ is reduced or not does not matter because of Proposition \ref{prop:SteinReduction}). Indeed, any irreducible component~$S$ of $\Supp F$ is contained in an irreducible component $Y$ of $X$. According to Proposition~\ref{lem:FiniteOverSteinAreStein}, $S$ is cohomologically Stein being a closed analytic subset of $Y$ (which is cohomologically Stein). Applying the inductive hypothesis to $\Supp F$ one concludes that it is cohomologically Stein.

Let $\{ X^{(i)}\}_{i \in I}$ be the set of irreducible components of $X$ and
$$ \pi \colon X' = \bigsqcup_{i \in I} X^{(i)} \too X$$
the natural map. The $k$-analytic space $X'$ is cohomologically Stein.  The morphism~$\pi$ is finite: indeed, this property being local on the base, one may suppose $X$ to be affinoid, for which the statement is clear (because an affinoid has at most finitely many irreducible components).

Let $\pi^\sharp \colon \Oc_X \to \pi_\ast \Oc_{X'}$ be the homomorphism induced by $\pi$ and consider the short exact sequence
$$ 0 \too \Oc_X \too \pi_\ast \Oc_{X'} \too C  \too 0,$$
where $C = \coker{\pi^\sharp}$. Let $U \subset X$ be the subset of points belonging to only one irreducible component. Its complement $X \smallsetminus U$ is a closed analytic subset of dimension $\le d- 1$.

Let $F$ be a coherent sheaf of $\Oc_X$-modules. Applying $\FHom(-, F)$ to the preceding short exact sequence, one obtains an exact sequence
$$ 0 \too \FHom(C, F) \too \FHom(\pi_\ast \Oc_{X'}, F) \stackrel{\phi}{\too} \FHom(\Oc_X, F) = F. $$
Observe the following:
\begin{enumerate}
\item The sheaf $\FHom(\pi_\ast \Oc_{X'}, F)$ has a natural structure of $\pi_{\ast} \Oc_{X'}$-module, hence it is the push-forward $\pi_\ast E$ of a coherent sheaf $E$ of $\Oc_{X'}$. Since $\pi$ is finite, according to Lemma \ref{lem:ComputationCohomologyFiniteMorphism},
$$ \H^q(X, \FHom(\pi_\ast \Oc_{X'}, F)) = \H^q(X', E) = 0,$$
for all $q \ge 1$.

\item The support of $C$ is contained in $X \smallsetminus U$ as well as the support $\FHom(C, F)$. Thus the higher cohomology of $\FHom(C, F)$ vanishes.
\end{enumerate}

Applying these considerations to the long exact sequence of cohomology groups associated to the short exact sequence,
$$ 0 \too \FHom(C, F) \too \FHom(\pi_\ast \Oc_{X'}, F) \too \Im(\phi) \too 0,$$
one obtains that the higher cohomology of $\Im(\phi)$ vanishes.  Consider the short exact sequence
$$ 0 \too \Im(\phi) \too F \too \coker(\phi) \too 0.$$

Since $\pi^\sharp$ is an isomorphism on $U$, the support of $\coker(\phi)$ is contained in $X \smallsetminus U$. Thus the higher cohomology of $\coker(\phi)$ vanishes. The long exact sequence of cohomology associated to the preceding short exact sequence gives $\H^q(X, F) = 0$ for all $q \ge 1$. 
\end{proof}

\begin{proposition} \label{prop:CohSteinFiniteCovering} Suppose $k$ \emph{non-trivially} valued. Let $X, Y$ be separated $k$-analytic spaces and $f \colon Y \to X$ a finite morphism. If $f$ is surjective and $Y$ is cohomologically Stein, then $X$ is cohomologically Stein.
\end{proposition}

\begin{proof} We adapt the argument given in \cite{LiuCRAS}. According to Proposition \ref{prop:CriterionCohSteinIrreducibleComponents} it suffices to prove the statement when $X$ is also assumed irreducible.

We argue by induction on the dimension $d$ of $X$. If $d = 0$ the statement is true. Suppose $d \ge 1$. If a coherent sheaf $F$ is not supported at the whole $X$, then its higher cohomology vanishes. Indeed, let $S$ be an irreducible component of $\Supp(F)$. Then $f^{-1}(S) \subset Y$ is cohomologically Stein because it is a closed analytic subset of $Y$ which is assumed to be cohomologically Stein. Since $\dim S \le d - 1$, the induction hypothesis implies that $S$ is cohomologically Stein. In particular, all the irreducible components of $\Supp(F)$ are cohomologically Stein, hence $\Supp F$ is cohomologically thanks to Proposition \ref{prop:CriterionCohSteinIrreducibleComponents}.

\begin{claim} Given $x \in X$, there are global sections $s_1, \dots, s_n \in \Oc(Y)$ such that the induced homomorphism $s \colon \Oc_X^n \to f_\ast \Oc_Y$ is an isomorphism at $x$.
\end{claim}

\begin{proof}[Proof of the Claim] It suffices to prove that there are sections $s_1, \dots, s_n$ such that $s \colon \Oc_X^n \to f_\ast \Oc_Y$ is surjective at $x$. The argument to prove this is similar to the one of Theorem A (Proposition \ref{prop:ThmA}), and we sketch the changes one has to perform. Begin with:

\begin{remark} \label{Rmk:ProofClaimFiniteStein}
Suppose $x$ belongs to a closed $k$-analytic subspace $Z$ of $X$ such that there are sections $s_1, \dots, s_n$ of $f_\ast \Oc_Y$ over $Z$ (\emph{i.e.} global sections of $\Oc_Y$ over $f^{-1}(Z)$) such that the induced homomorphism $\Oc_Z^n \to (f_\ast \Oc_Y)_{\rvert Z}$ is surjective at $x$. Since $Y$~is cohomologically Stein, the sections $s_1, \dots, s_n$ extend to global sections of~$\Oc_Y$ over~$Y$ and the induced map $\Oc_X^n \to f_\ast \Oc_Y$ is surjective at $x$.
\end{remark}

As for the proof of Proposition \ref{prop:ThmA}, one may suppose $X$ irreducible and argue on the dimension of $X$. Again for $d = 0$ there is nothing to prove, so one can assume $d$ positive and that the statement is true for irreducible space of lower dimension. If $x$ is not Zariski-dense, the result holds by induction hypothesis and Remark \ref{Rmk:ProofClaimFiniteStein}. It remains to treat the case when $x$ is Zariski-dense. If $X$ admits a closed analytic subspace $Z$, the  induction hypothesis yields a homomorphism $s \colon \Oc_X^n \to f_\ast \Oc_Y$ which is surjective at some point of $Z$. Since the locus where $s$ is surjective is a Zariski-open subset of $X$, $s$ is surjective at $x$.  If $X$ has no closed analytic subspace, then Proposition \ref{Prop:ZariskiTrivialSpaces} states that $X = \M(A)$ for a local finite algebra $A$ over an analytic extension $K$ of $k$. Then $Y = \M(B)$ for a finite $K$-algebra $B$ and the result holds trivially.
\end{proof}

The rest of the proof goes along the same lines as that of Proposition \ref{prop:CriterionCohSteinIrreducibleComponents}, applying $\FHom(-, F)$ to the exact sequence
\begin{equation*} \Oc_X^n \too f_\ast \Oc_{Y} \too \coker(s)  \too 0. \qedhere \end{equation*}
\end{proof}

\section{Preliminaries on Liu spaces}\label{sec:Liu}

\subsection{Recollecting some results of Liu} Let $X$ be a $k$-analytic space and $x \in X$ a rigid point. Let $\mathfrak{m}_x$ be the kernel of the evaluation map $\Oc(X) \to k(x)$, where~$k(x)$ denotes the residue field at $x$. The ideal $\mathfrak{m}_x$ is maximal: indeed, by definition the integral domain $\Oc(X) / \mathfrak{m}_x$ injects into $k(x)$ which is finite-dimensional over $k$, and it follows that $\Oc(X) / \mathfrak{m}_x$ is a field. For a ring $A$, let $\Max(A)$ denote the set of maximal ideals of $A$.

\begin{lemma} A $k$-analytic space $X$ is rig-holomorphically separable if and only if the map $X_{\rig} \to \Max(\Oc(X))$,  $x \mapsto \mathfrak{m}_x$ is injective.
\end{lemma}

\begin{lemma} \label{lem:RigHolSepImpliesHolSep} Let $X$ be $k$-analytic space, $x$, $x'$ distinct points of $X$ and $f$ a holomorphic function on $X$ such that $|f(x)| \neq |f(x')|$. Then,
\begin{enumerate}
\item if $x$ is rigid, there is $g \in \Oc(X)$ such that $g(x) = 0$ and $|g(x')| > 0$;
\item if $x, x'$ are both rigid, there is $g \in \Oc(X)$ such that $g(x) = 0$ and $g(x') = 1$.
\end{enumerate}
In particular, if $X$ is holomorphically separable, then it is rig-holomorphically separable. 
\end{lemma}

\begin{proof} (1) Let $P$ be the minimal polynomial of $f(x)$ over $k$. Then $P(f(x)) = 0$ and $P(f(x')) \neq 0$ (otherwise $f(x')$ would have the same norm as $f(x)$). 

(2) According to (1) there is $g \in \Oc(X)$ such that $g(x) = 0$ and $g(x') \neq 0$. Let $P$ be the minimal polynomial of $g(x')$ over $k$. Then $P(g(x')) = 0$ and $P(g(x)) = P(0)$ is not zero. The function $\frac{P(g)}{P(0)}$ does the job.
\end{proof}

\begin{remark} According to the preceding lemma, an $S$-space in the sense of Liu (\cite[D\'efinition~2.2]{LiuTohoku}) is a strict, separated, compact, rig-holomorphically separable $k$-analytic space such that $\Oc_X$ is acyclic.
\end{remark} 

\begin{theorem} \label{thm:LiuGerritzenGrauert} Suppose $k$ is \emph{non-trivially} valued. Let $X$ be a strict, separated, compact, rig-holomorphically separable $k$-analytic space such that $\Oc_X$ is acyclic. 
Then every strictly affinoid domain $D$ of $X$ is a finite union of rational domains of $X$.
\end{theorem}

\begin{proof} This is Liu's version of Gerritzen-Grauert Theorem for $S$-spaces \cite[Th\'eor\`eme 1]{LiuTohoku}.
\end{proof}

\begin{corollary} \label{Cor:StrictLiuSpacesAreHolSeparable}With the hypotheses of Theorem \ref{thm:LiuGerritzenGrauert}, $X$ is holomorphically separable hence a Liu space.
\end{corollary}

\begin{proof} Given two distinct points $x, x'$ of $X$, according to Theorem \ref{thm:LiuGerritzenGrauert} there is a rational domain $D$ of $X$ containing $x$ and not containing $x'$. In particular there is $f \in \Oc(X)$ such that $|f(x)| \neq |f(x')|$.
\end{proof}

\subsection{Results for non-strict spaces}

\subsubsection{Rational domains}

\begin{definition} Let $X$ be a $k$-analytic space, $f_0, \dots, f_n \in \Oc(X)$ without common zeros and $r_1, \dots, r_n > 0$. The \emph{rational domain} associated with this data is the analytic domain
$$ \{ x \in X : |f_i(x)| \le r_i |f_0(x)|, i = 1, \dots, n\}.$$
A rational domain of $X$ is an analytic domain of the previous form.
\end{definition}

\begin{lemma}\label{Lemma:BaseChangeWithDisc} Let $X$ be a $k$-analytic space that is separated and countable at infinity. Let~$F$ be a coherent sheaf on $X$ that is universally acyclic. Let $\D := \D(r_1, \dots, r_n)$ be the disc of radii $r_1, \dots, r_n >0 $ and $p \colon X \times_k \D \to X$ the first projection. Then, 
\begin{enumerate}
\item $F(X)[t_1, \dots, t_n]$ is dense in $p^\ast F(X \times_k \D)$;
\item $p^\ast F$ is universally acyclic on $X \times_k \D$.
\end{enumerate}
\end{lemma}

\begin{proof} Both results follow from Theorem~\ref{thm:FXGY}. \end{proof}

\begin{lemma} \label{lem:PrincipalOpenSubsetSSpace} Let $X$ be a $k$-analytic space that is separated and countable at infinity and on which $\Oc_X$~is universally acyclic. Let $f \in \Oc(X)$ and $U = \{ x \in X : f(x) \neq 0\}$. Then,
\begin{enumerate}
\item if $F$ is a universally acyclic coherent sheaf on $X$, then $F_{\rvert U}$ is universally acyclic on $U$;
\item the image of $F(X)[f^{-1}]$ is dense in $F(U)$.
\end{enumerate}
\end{lemma}

\begin{proof} Since $X$ is separated, $U$ is identified with the closed subspace of $X \times \A^{1, \an}_k$ associated with the sheaf of ideals $I$ generated by $t f - 1$, where $t$ is a coordinate function on $\A^{1, \an}_k$. Let $p \colon X \times \A^{1, \an}_k \to X$ be the first projection. For a coherent sheaf $F$ on $X$, the short exact sequence of coherent sheaves,
$$ 0 \too I p^\ast F \too p^\ast F \too F_{\rvert U} \too 0,$$
induces a long exact sequence
$$ \cdots \too \H^q(X \times_k \A^{1, \an}_k, p^\ast F) \too \H^q(U, F) \too \H^{q+1}(X \times_k \A^{1, \an}_k, I p^\ast F) \too \cdots .$$
The function $tf - 1$ is not a zero-divisor on $X \times \A^{1, \an}_k$, thus $I p^\ast F$ is isomorphic to $p^\ast F$ as a coherent sheaf. According to Theorem~\ref{thm:FXGY}, for all $q \ge 1$, we have
$$ \H^q(X \times_k \A^{1, \an}_k, p^\ast F) = \H^q(X, F) = 0.$$
In particular $F_{\rvert U}$ is acyclic, and even universally acyclic since the assumptions of the statement are stable under extension of scalars. Moreover, the restriction map $ p^\ast F(X \times \A^{1, \an}_k) \to F(U)$ is surjective. Since $F(X)[t]$ is dense in $p^\ast F(X \times \A^{1, \an}_k)$ by Theorem \ref{thm:FXGY}, the map $F(X)[t] \to F(U) $ sending $t$ to $f^{-1}$ has dense image, whence the statement.
\end{proof}

\begin{proposition} \label{prop:FunctionsOnRationalDomains} Let $X$ be a $k$-analytic space that is separated and countable at infinity and on which $\Oc_X$ is universally acyclic. Let $f_0, \dots, f_n \in \Oc_X(X)$ without common zeros,  $r_1, \dots, r_n > 0$ and
$$ X' = \{ x \in X : |f_i(x)| \le r_i |f_0(x)| \textup{ for all }i = 1, \dots, n\}.$$
Then,
\begin{enumerate}
\item if $F$ is a universally acyclic coherent sheaf on $X$, then $F_{\rvert X'}$ is universally acyclic on $X'$;
\item the image of $F(X)[f_0^{-1}]$ is dense in $F(X')$.
\end{enumerate}

\end{proposition}

\begin{proof} Since the functions $f_0, \dots, f_n$ do not have common zeros, the subspace $X'$ is contained in the open subset $U = \{ x \in X : f_0(x) \neq 0 \}$. According to Lemma \ref{lem:PrincipalOpenSubsetSSpace}, the open subset $U$ is separated and $\Oc_U$ is universally acyclic. By replacing $X$ by $U$, the function $f_0$ can be assumed to be invertible and, up to replacing $f_i$ by $f_i / f_0$, equal to $1$. When $f_0 = 1$, arguing by induction, one reduces to the case $n  =1$. In this situation write $f = f_1$ and $r = r_1$. The statement can then be proved by reasoning as for Lemma \ref{lem:PrincipalOpenSubsetSSpace} replacing the analytic affine line~$\A^{1, \an}_k$ by the closed disc with center~0 and radius~$r$. 
\end{proof}

\begin{lemma} \label{prop:RationalNeighbourhoodSSpace} Let $X$ be a holomorphically separable $k$-analytic space. For $x_0 \in X$ and  a compact analytic neighbourhood~$D$ of~$x_{0}$, there are $f_0, \dots, f_n \in \Oc_{X}(X)$ without common zeros and $r_1, \dots, r_n > 0$ such that the domain
$$X' = \{ x \in X : |f_i(x)| \le r_i |f_0(x)|, i = 1, \dots, n \},$$
satisfies the following properties:
\begin{enumerate}
\item $D' := X' \cap D$ is contained in the topological interior of $D$ in $X$;
\item $x_0$ belongs to $D'$;
\item $D'$ is a finite union of connected components of $X'$.
\end{enumerate}
\end{lemma}

\begin{proof} Since the space $X$ is holomorphically separable, for every $y$ in the topological boundary $\partial D$ of $D$ in $X$, there is a function $f \in \Oc_{X}(X)$ such that $|f(x)| \neq |f(y)|$. By compactness of $\partial D$, there are functions $g_1, \dots, g_\ell, h_1, \dots, h_m \in \Oc_{X}(X)$ and real numbers $\alpha_1, \dots, \alpha_\ell, \beta_1, \dots, \beta_m > 0$ such that
\begin{align*}
|g_i(x_0)| \le \alpha_i, \quad i = 1, \dots, \ell, && |h_j(x_0)| \ge \beta_j, \quad j = 1, \dots, m,
\end{align*}
and, for all $x \in \partial D$, we have either $|g_{i}(x)|>\alpha_{i}$ for some $i \in \{1,\dotsc,\ell\}$ or $|h_{j}(x)| < \beta_{j}$ for some $j \in \{1,\dotsc,m\}$.

Consider the subspace $X'$ of $X$ made of the points $x \in X$ such that
\begin{align*}
|g_i(x)| \le \alpha_i, \quad i = 1, \dots, \ell, && |h_j(x)| \ge \beta_j, \quad j = 1, \dots, m,
\end{align*}
Then $D' := X' \cap D$ contains $x_0$ and is contained in the topological interior of $D$. It follows that $D'$ is a finite union of connected components of $X'$. Moreover $X'$ can be put in the form in the statement by setting
\begin{align*}
f_0 &= h_1 \cdots h_m, \\
f_i &= g_i f_0, &r_i& = \alpha_i & (i = 1, \dots, \ell), \\
f_{\ell +i} &= h_1 \cdots h_{i -1} h_{i + 1} \cdots h_m, &r_{\ell+i}& = \beta_i^{-1} & (i = 1, \dots, m), \\
f_{\ell + m + 1} &= 1, &r_{\ell+m+1}& = (\beta_1 \cdots \beta_m)^{-1}.
\end{align*} 
This finishes the proof.
\end{proof}

\begin{definition}  \label{def:Definitionk_r} Let $r = (r_1, \dots, r_N) \in \R_{>0}^n$.
\begin{enumerate}
\item  The field $k_r$ is the completion of the field of rational fractions $k(t_1, \dots, t_N)$ with respect to the absolute value 
\[\sum_{\ell \in \N^N} a_{\ell} t^\ell \in k[t_{1},\dotsc,t_{N}] \mapsto \max_{\ell\in \N^N} \{|a_\ell | r^\ell\} \in \R_{>0},\] 
where $t^\ell := t_1^{\ell_1} \dotsb t_N^{\ell_N}$ and $r^\ell := r_1^{\ell_1} \dotsb r_N^{\ell_N}$.
\item The real numbers $r_1, \dots, r_N$ are said to be \emph{free over $|k^\times|$} if their images in the $\Q$-vector space $\R_{>0} /(|k^\times|\otimes_{\Z} \Q)$ are free.
\end{enumerate}
\end{definition}

Assume that the real numbers $r_1, \dots, r_N$ are free over $|k^\times|$. Then the field $k_r$ can be obtained as the affinoid $k$-algebra
\[ k \{ r_1^{-1} t_1, \dots, r_N^{-1}t_N, r_1 u_1, \dots, r_N u_N\} / (t_i u_i - 1, i = 1, \dotsc, N). \]
In more concrete terms, it may be described as the field of power series of the form
\[ f = \sum_{\ell \in \Z^N} a_{\ell} t^\ell, \]
where $a_{\ell} \in k$ and the family $(|a_{\ell}| r^\ell)_{\ell\in \Z^N}$ is summable. It is endowed with the absolute value~$\va_{r}$ for which $\|f\|_{r} = \max_{\ell\in\Z^N} \{|a_{\ell}| r^\ell\}$.

More generally, if $(A,\nm)$ is a Banach $k$-algebra, then $A\hat{\otimes}_{k} k_{r}$ is a Banach $k_{r}$-algebra that may be explicitly described as the set of power series of the form
\[ f = \sum_{\ell \in \Z^N} a_{\ell} t^\ell, \]
where $a_{\ell} \in A$ and the family $(\|a_{\ell}\| r^\ell)_{\ell\in \Z^N}$ is summable. It is endowed with the norm~$\nm_{r}$ for which $\|f\|_{r} = \max_{\ell\in\Z^N} \{\|a_{\ell}\| r^\ell\}$.

\begin{proposition} \label{prop:LiuGerritzenGrauertNonStrict} Let~$X$ be a Liu space.
Then every affinoid domain of $X$ is a finite union of rational domains of $X$. In particular, $X$ is a finite union of rational affinoid domains. 
\end{proposition}

\begin{proof} We follow Ducros's argument to draw the Gerritzen-Grauert theorem for affinoid spaces from the corresponding result in the strictly affinoid case (\emph{cf.} \cite[Lemme 2.4]{DucrosSemiAlgebrique}). We may assume that~$X$ is reduced. We will endow $\Oc(X)$ with the spectral norm, and similarly for the algebras of the affinoid domains of~$X$ that will appear. Since $X$ is compact, there are positive real numbers $r_1, \dots, r_N$, linearly independent in the $\Q$-vector space $\R_{>0} / ( |k^\times| \otimes \Q )$, such that the $k_r$-analytic space $X_{k_r}$ is strict. Note that we have $\Oc(X_{k_{r}}) = \Oc(X) \hat{\otimes}_{k} k_{r}$.

The natural projection $\pr \colon X_{k_r} \to X$ has a (continuous) section~$\sigma$, defined as follows: given an affinoid domain $D = \M(A)$ of~$X$ and a point $x\in D$, $\sigma(x)$ is the point associated to the multiplicative seminorm 
\[f = \sum_{\ell\in \Z^n} a_\ell t^\ell \in A \hat{\otimes}_{k} k_{r} \mapsto \max_{\ell\in \Z^n} \{|a_\ell(x)| r^\ell\} \in \R_{>0}.\]

According to Theorem \ref{thm:LiuGerritzenGrauert} the affinoid domain $D_{k_r}$ is a finite union of rational domains of $X_{k_r}$. Therefore it suffices to prove the following: given a rational domain~$Y$ of $X_{k_r}$, $\sigma^{-1}(Y)$ is a finite union of rational domains of $X$. Let $f_0, \dots, f_n \in \Oc(X_{k_r})$ be without common zeroes, $s_1, \dots, s_n > 0$ and consider the rational domain
$$ Y = \{ x \in X_{k_r} : |f_i(x)| \le s_i |f_0(x)|, i = 1, \dots, n\}.$$

For each $i\in \{1,\dotsc,n\}$, write $f_i = \sum_{\ell \in \Z^N} f_{i \ell} t^\ell$ with $f_{i \ell} \in \Oc(X)$. Since the functions $f_0, \dots, f_n$ do not have a common zero, the function $f_0$ is necessarily invertible on $Y$ and, because of the compactness of the latter, there is a positive real number $\epsilon > 0$ such that $|f_0(y)| \ge \epsilon$ for all $y \in Y$. Let $L \subset \Z^N$ be a finite subset such that, for $\ell \not \in L$ and $i\in \{1,\dotsc,n\}$, the inequalities $\| f_{0,\ell} \|r^\ell < \epsilon$ and $\| f_{i,\ell} \|r^\ell < s_i \epsilon$ hold.

A point $x \in X$ then belongs to $\sigma^{-1}(Y)$ if and only if the following conditions are satisfied:
\begin{itemize}
\item there is $\ell \in L$ such that $|f_{0, \ell}(x)| r^\ell \ge \epsilon$;
\item $\displaystyle   \max_{\ell \in L} \{ |f_{i, \ell}(x)| r^\ell \} \le s_i  \max_{\ell \in L} \{ |f_{0, \ell}(x)| r^\ell \} $ for $i = 1, \dots, n$.
\end{itemize}
Fix $\ell \in L$. Let $Z_\ell$ be the rational domain of $X$ made of the points $x \in X$ such that
\begin{align*} 
|f_{0, \ell}(x)| r^\ell &\ge \epsilon, \\
 |f_{0, \ell}(x)| r^\ell &\ge |f_{0, m}(x)| r^m, &(m \in L \smallsetminus \{ \ell\}),  \\
s_i |f_{0, \ell}(x)| r^\ell  &\ge |f_{i, m}(x)| r^m, &(m \in L , i = 1, \dots, n).
\end{align*}
By construction $\sigma^{-1}(Y)$ is the union of the rational domains $Z_{\ell}$ for $\ell \in L$, which concludes the proof.
\end{proof}

Let~$X$ be a Liu space and write it as a finite union of rational affinoid domains as in the proposition. By Proposition~\ref{prop:FunctionsOnRationalDomains}, the algebra of each of those affinoid domains may be topologically generated by global functions. Putting all those functions together, we define a morphism from a Tate algebra to~$\Oc_{X}(X)$. The precise outcome is the following result. Its proof is similar to that of~\cite[Proposition~3.3]{LiuTohoku} and we omit it.

\begin{proposition}
Let~$X$ be a $k$-analytic Liu space. Then, there exists a $k$-affinoid space~$Y$ and a locally closed immersion $\varphi \colon X \to Y$. More precisely, there exists a finite covering $\{Y_{i}\}_{i}$ by rational domains of~$Y$ such that, for each~$i$, the morphism $\varphi^{-1}(Y_{i}) \to Y_{i}$ is a Runge immersion, \emph{i.e.} the composition of a closed immersion and a Weierstrass domain embedding.
\qed 
\end{proposition}

\subsubsection{Extension of scalars}

\begin{proposition} \label{Prop:SSpaceStableScalarExt} Let $X$ be a separated, countable at infinity, holomorphically separable $k$-analytic space on which $\Oc_X$ is universally acyclic. For every analytic extension $k'$ of~$k$, the $k'$-analytic space $X_{k'}$ is holomorphically separable and $\Oc_{X_{k'}}$ is universally acyclic.
\end{proposition}

\begin{proof} The sheaf $\Oc_{X_{k'}}$ is universally acyclic by Theorem~\ref{thm:extensioncohomology}. It remains to prove that $X_{k'}$ is holomorphically separable. If the projection of two points $x, x' \in X_{k'}$ in $X$ are distinct, by hypothesis there is $f \in \Oc(X)$ such that $|f(x)| \neq |f(x')|$. Otherwise call $y$ the projection of $x, x'$ in $X$. 

Suppose first that $X$ is compact and let $D$ be affinoid domain of $X$ containing~$y$. By Proposition \ref{prop:LiuGerritzenGrauertNonStrict}, we can suppose that $D$ is rational in~$X$. In particular there is $f \in \Oc(X)$ such that $\Oc(X)[f^{-1}]$ is dense in $\Oc(D)$ and $\Oc(X)[f^{-1}] \otimes k'$ is dense in $\Oc(D_{k'}) = \Oc(D)\hotimes k'$. Since $D_{k'}$ is affinoid, there exists $g \in \Oc(D_{k'})$ such that $|g(x)| \neq |g(x')|$. By density such a $g$ can be taken in $\Oc(X)[f^{-1}] \otimes k'$.  By clearing denominators one produces a holomorphic function $h \in \Oc(X)$ such that $|h(x)| \neq |h(x')|$.

In the case where $X$ is not necessarily compact, according to Lemma \ref{prop:RationalNeighbourhoodSSpace} and Proposition \ref{prop:FunctionsOnRationalDomains}, there is a compact analytic domain~$D'$ of~$X$ containing~$y$ such that $\Oc_{D'}$~is universally acyclic and there is a function $f \in \Oc(X)$ with no poles on~$D'$ such that $\Oc(X)[f^{-1}]$ is dense in $\Oc(D')$. This leads us back to the compact case.
\end{proof}

\begin{corollary} Let~$X$ be a $k$-analytic Liu space. For every analytic extension~$k'$ of~$k$, $X_{k'}$ is a Liu space.
\qed
\end{corollary}

\begin{corollary} \label{Cor:LiuSpaceCohomStein} Liu spaces are cohomologically Stein. \end{corollary}

\begin{proof} Given a $k$-analytic Liu space $X$, pick a non-trivially valued extension $k'$ of~$k$ such that $X_{k'}$ is strict. Since $X_{k'}$ is again a Liu space, Liu's theorem (Theorem \ref{Thm:LiuTohoku}) states that $X_{k'}$ is cohomologically Stein. By applying Theorem \ref{thm:extensioncohomology} one concludes that $X$ is cohomologically Stein.
\end{proof}

\begin{corollary} \label{cor:SpectrumLiuSpace} For a Liu space $X$ the natural map $X \to \M(\Oc(X))$ is a homeomorphism.
\end{corollary}

\begin{proof} Since the map is continuous and the source and the target are compact and Hausdorff topological spaces, it suffices to prove that the map is bijective. By definition, a Liu space is holomorphically separated, thus the map is injective. In order to prove that it is surjective, let $x$ be a multiplicative seminorm on $\Oc(X)$. 

Let $k'$ be the completion of the fraction field of the integral domain $\Oc(X) / \ker x$. By definition the homomorphism $\Oc(X) \to k'$ of Banach $k$-algebras is bounded and thus extends to a continuous homomorphism of Banach $k'$-algebras $\Oc(X) \hotimes_k k' \to k'$. The Banach $k'$-algebra $\Oc(X) \hotimes_k k'$ is identified with $\Oc(X_{k'})$ thanks to Theorem~\ref{thm:extensioncohomology}. It follows that $x$ extends to a bounded multiplicative seminorm $x'$ on~$\Oc(X_{k'})$. If the seminorm $x'$ comes from a point of $X_{k'}$, then $x$ comes from a point of $X$. 

Therefore, up to extending scalars, one may assume $k' = k$. Moreover, up to extending further $k$, one may assume that it is non-trivially valued and that the space $X$ is strict. Liu proved that rigid points of a strict Liu space $Y$ are in bijection with maximal ideals of $\Oc(Y)$ (see \cite[Proposition 2.4 or 3.2]{LiuTohoku}). This concludes the proof.
\end{proof}

\section{Descent of holomorphic separability and convexity} \label{sec:HolomorphicConvexity} 

\subsection{Some normed algebra}The proofs of Propositions  \ref{Prop:CompatibilityHolConvexhullsExtScalars} and \ref{Prop:CompatibilityHolConvexhullsProducts} rely on some results of normed algebra that we review in this section.

\begin{definition}\
\begin{enumerate}

\item A \emph{normed $k$-vector space} is a $k$-vector space $V$ equipped with a non-archimedean norm $\| \cdot \|_V$ such that $\| \lambda v\|_V = |\lambda| \| v \|_V$ for all $\lambda \in k$ and $v \in V$. 

\item Let $\alpha > 1$ and $V$ be a normed $k$-vector space. A family of vectors $\{ v_i\}_{i \in I}$ of $V$ is \emph{$\alpha$-cartesian} if, for every finite subset $J \subset I$, $\lambda_i \in k$, $v_i \in V$ for $i \in J$,
$$ \alpha^{-1}\max_{i \in J} \{ |\lambda_i| \| v_i\|_V \} \le \| v \|_V \le \max_{i \in J} \{ |\lambda_i| \| v_i\|_V \},$$
where $v = \sum_{i \in J} \lambda_i v_i$.

\item Let $V, W$ be normed $k$-vector spaces. The tensor product $V \otimes_k W$ is endowed with the tensor norm $\| \cdot \|_{V \otimes W}$ defined, for $t \in V \otimes_k W$, as
$$ \| t \|_{V \otimes W} = \inf \max_{i = 1, \dots, n} \{ \| v_i\|_V \| w_i \|_W \},$$
the infimum ranging on the set of possible ways of writing $t$ as  $\sum_{i = 1}^N v_i \otimes w_i$ with $N \in \N$ and $v_i \in V$, $w_i \in W$ for all $i = 1 \dots, N$.
\end{enumerate}
\end{definition}

\begin{proposition} \label{prop:ExistenceAlphaCartesianBasis} Let $V$ be a finite-dimensional normed $k$-vector space and $\alpha > 1$. Then there exists an $\alpha$-cartesian basis of $V$.
\end{proposition}

\begin{proof} If $k$ is non-trivially valued, this is \cite[2.6.2 Proposition 3]{BGR}\footnote{If $k$ is maximally complete the statement holds also with $\alpha = 1$ (\cite[Proposition 1.1]{GoldmanIwahori}, \cite[2.4.4 Proposition 2]{BGR}).}. 
The proof of \cite[Proposition 1.1]{GoldmanIwahori} works also in the trivially valued case. Let us sketch it here. One argues by induction on the dimension $d$ of $V$. If $d = 0$ the statement is clearly true. Suppose $d \ge 1$ and consider a linear form $\phi \colon V \to k$. Then there exists $v_1 \in V$ non-zero such that, for all non-zero $v \in V$,
$$ \frac{|\phi(v_1)|}{\| v_1\|_V} \ge \alpha^{-1} \frac{|\phi(v)|}{\| v\|_V}. $$

One concludes the proof applying the inductive hypothesis to $W = \ker \phi$.
\end{proof}

\begin{proposition} \label{Prop:CompatibiliyAlphaCartesianBasisTensorProduct} Let $V$, $W$ be normed $k$-vector spaces. Let $\alpha>1$ and $v_1, \dots, v_n$ be a linearly independent $\alpha$-cartesian family of $V$. Then, for $w_1, \dots, w_n \in W$,
$$ \alpha^{-1}\max_{i= 1, \dots, n} \{ \| v_i\|_V \| w_i\|_W  \} \le \| t \|_{V \otimes W} \le \max_{i= 1, \dots, n} \{ \| v_i\|_V \| w_i \|_W \},$$
where $t = \sum_{i = 1}^n v_i \otimes w_i$. 
\end{proposition}

\begin{proof} Let $V_0$ be the sub-$k$-vector space of~$V$ generated by $v_{1},\dots, v_n$. According to \cite[Lemme~3.1]{PoineauAngeliques} the inclusion $V_0 \otimes_{k} W \to V \otimes_{k} W$ is an isometry. Therefore, up to replacing~$V$ by~$V_{0}$, one may assume that $V$ is finite dimensional and $v_1, \dots, v_n$ is a basis of $V$.

Let $w_1, \dots, w_n \in W$ and $t = \sum_{i = 1}^n v_i \otimes w_i$. By definition of the tensor norm, for $\epsilon > 0$, there exist an integer $n' \ge 0$, $v'_{1},\dots,v'_{n'} \in V$ and $w'_{1},\dots,w'_{n'} \in W$ such that $t = \sum_{i=1}^{n'} v'_{i} \otimes w'_{i}$ and 
$$ \| t \|_{V \otimes W} \ge \max_{i =1, \dots, n'} \{ \| v'_i\|_V \| w'_i \|_W \} - \epsilon.$$
For $i = 1, \dots, n'$, write $ v'_i = \sum_{j = 1}^n \lambda_{ij} v_i$, with $\lambda_{ij} \in k$. Since the basis $v_1, \dots, v_n$ is $\alpha$-cartesian,
\begin{align*} 
\| t \|_{V \otimes W} &\ge \alpha^{-1} \max_{i = 1, \dots, n'} \left\{  \max_{j = 1, \dots, n} \{  |\lambda_{ij}| \| v_i \|_V \} \| w_i' \|_W \right\} - \epsilon \\
&= \alpha^{-1} \max_{j = 1, \dots, n}\left\{  \max_{i = 1, \dots, n'} \{  |\lambda_{ij}| \| w_i' \|_W \} \| v_i \|_V \right\} - \epsilon \\
&\ge \alpha^{-1} \max_{j = 1, \dots, n}\left\{   \| w_j'' \|_W \| v_i \|_V \right\} - \epsilon,
\end{align*}
where $w_j'' = \sum_{i = 1}^n \lambda_{ij} w'_i$. Since $v_1, \dots, v_n$ is a basis and
$$ \sum_{j = 1}^n w_j'' \otimes v_j = t= \sum_{j = 1}^n w_j \otimes v_j,$$
one has $w_j'' = w_j$ for all $j = 1, \dots, n$. The result follows letting $\epsilon$ tend to $0$.
\end{proof}

\subsection{Holomorphic convex hull and extension of scalars}

We prove here that the formation of the holomorphic convex hull is compatible with extension of scalars:

\begin{proposition} \label{Prop:CompatibilityHolConvexhullsExtScalars} Let $X$ be a $k$-analytic space that is countable at infinity and $K$~be a compact subset of $X$. Let~$k'$ be an analytic extension of $k$, $X' = X_{k'}$ and $K' = \pr_{k'/k}^{-1}(K)$. Then,
$$ \hat{K}'_{X'} = \pr_{k'/k}^{-1}(\hat{K}_X).$$
\end{proposition}

As an immediate consequence, we have:

\begin{corollary} Let $X$ be a $k$-analytic space, countable at infinity. Then, $X$ is holomorphically convex if and only if $X_{k'}$ is holomorphically convex for some analytic extension $k'$ of $k$.
\end{corollary}

The compatibility in Proposition \ref{Prop:CompatibilityHolConvexhullsExtScalars} relies on a compatibility of the spectral norm with extension of scalars. To state it let us define:

\begin{definition} \label{def:RingGermsAlongCompact} Let $X$ be a $k$-analytic space and $K$ a compact subset of $X$. Let $\cal{D}_K$ be the set of compact analytic domains containing $K$.
 The \emph{ring of germs along $K$} is the $k$-algebra
$$ A_K := \varinjlim_{D \in \cal{D}_K} \Oc(D).$$
The $k$-algebra $A_K$ is equipped with the sup norm $\| \cdot \|_{A_K}$ on the compact $K$.
\end{definition}

\begin{proposition} \label{prop:SpectralNormExtensionScalars} Let $X$ be a $k$-analytic space, $K \subset X$ a compact subset and $A$ the ring of germs along $K$. Let $k'$ be an analytic extension of $k$, $\pr\colon X' = X \times_k k' \to X$ be the morphism induced by extension on scalars and $K' := \pr^{-1}(K)$. Denote by $\| \cdot \|_{A, k'}$ the tensor norm on $A \otimes_k k'$ and $\rho_{A,k'}$ the associated spectral norm. Then, for every $f \in A \otimes_k k'$,
$$ \rho_{A, k'}(f) := \lim_{n \to \infty} \| f \|_{A, k'}^{1/n} = \sup_{\pr^{-1}(K)} |f|. $$
\end{proposition}

\begin{proof} We prove first the inequality $ \rho_{A, k'}(f) \ge \sup_{\pr^{-1}(K)} |f|$.
Let $x \in K$ be a point. By definition the evaluation map $\ev_x \colon A \to \khat(x)$ is norm-decreasing, therefore the map induced on tensor products,
$$ \ev_x \otimes \id \colon A \otimes_k k' \too \khat(x) \otimes_k k',$$
is norm decreasing too. According to \cite[Theorem 1.3.1]{Berkovich90}, for $f \in A \otimes_k k'$,
$$ \rho_{A, k'}(f) \ge \rho_{\khat(x), k'}(f) = \sup_{\cal{M}(S)} |f|,$$
where $S$ is the completion of $\khat(x) \otimes_k k'$ with respect the tensor norm $\| \cdot \|_{\khat(x), k'}$ and $\rho_{\khat(x), k'}$ is the spectral norm associated to $\| \cdot \|_{\khat(x), k'}$. The spectrum $\cal{M}(S)$ is naturally identified with the fibre $\pr^{-1}(x)$ whence the wanted inequality.

On the other hand, suppose first that $K$ is an affinoid domain in $X$. Then $K = \cal{M}(A)$, $K' = \cal{M}(A \hotimes_k k')$ and, for $f \in A \otimes_k k'$,
$$ \rho_A(f) = \sup_{\pr^{-1}(K)} |f|.$$

Suppose then that $K$ is a finite union of affinoid domain $D_1, \dots, D_n$ in $X$ with $D_i = \cal{M}(A_i)$. Then, for $f \in A$,
$$ \rho_A(f) = \max \{ \rho_{A_1}(f), \dots, \rho_{A_n}(f) \},$$
where, for $i = 1, \dots, n$,  $\rho_{A_i}$ is the sup norm on $D_i$. In other words, endowing $A$ and $A_i$ with the corresponding spectral norms, the map
\begin{eqnarray*}
\Phi \colon A & \too & A_1 \times \cdots \times A_n \\
f & \longmapsto & (f_{\rvert D_1}, \dots, f_{\rvert D_n}),
\end{eqnarray*}
is isometric. It follows from \cite[Lemme 3.1]{PoineauAngeliques} that the induced map on the tensor products  $$\Phi \otimes_k k' \colon A \otimes_k k' \too (A_1 \otimes_k k') \times \cdots \times (A_n \otimes_k k'), $$
is isometric too. For $f \in A \otimes_k k'$ the equality $ \rho_{A, k'}(f) = \sup_{\pr^{-1}(K)} |f|$  follows from the preceding case. 

Suppose $K$ arbitrary. For $D \in \cal{D}_K$ the restriction map $\Oc(D) \to A$ is norm decreasing, thus the induced map $\Oc(D) \otimes_k k' \to A \otimes_k k'$ is also norm decreasing. It follows from the preceding case, for $ f \in \Oc(D)$,
$$ \rho_{A, k'}(f) \le \rho_{\Oc(D), k'}(f) = \sup_{\pr^{-1}(D)} |f|. $$

Fix $f \in A \otimes_k k'$ and let $D_0 \in \cal{D}_K$ be such that $f$ is the restriction of an element in $\Oc(D_0) \otimes_k k'$ that we denote again by $f$. Then,
$$ \rho_{A, k'}(f) \le \inf_{\substack{D \in \cal{D}_K \\ D \subset D_0}} \sup_{\pr^{-1}(D)} |f| = \sup_{K} |f|,$$
which concludes the proof.
\end{proof}

\begin{proof}[{Proof of Proposition \ref{Prop:CompatibilityHolConvexhullsExtScalars}}] Let us abbreviate $\pr_{k'/k}$ by $\pr$. The inclusion $\hat{K}'_{X'} \subset \pr^{-1}(\hat{K}_{X})$ is rather obvious. The converse inclusion is equivalent to the following statement:
for all $x \in \hat{K}_X$ and all $f \in \Oc(X')$,
$$  \| f \|_{\pr^{-1}(x)} := \sup_{\pr^{-1}(x)} |f| \le \| f \|_{K'} := \sup_{K'} |f|. $$
In order to prove it, consider the ring of germs along $K$ as defined in Definition \ref{def:RingGermsAlongCompact},
$$ A := \varinjlim_{D \in \cal{D}_K} \Oc(D)$$
endowed with the sup norm $\| \cdot \|_K$ on the compact $K$ ($\cal{D}_K$ is the family of analytic domains $D \subset X$ containing $K$).

\begin{lemma} Let $\alpha > 1$ and $g \in \Oc(X) \otimes_{k} k'$. Then,
$$
\sup_{\pr^{-1}(x)} |g| \le \alpha \| g \|_{A, k'}, 
$$
where $\| \cdot \|_{A, k'}$ denotes the tensor norm on $A \otimes_k k'$.
\end{lemma}

\begin{proof}[Proof of the Lemma] Let $V$ be a finite dimension $k$-vector space of $k'$ such that $g$ lies in $\Oc(X) \otimes_k V$. Since $V$ is finite dimensional there exists a $\alpha$-cartesian basis $\lambda_1, \dots, \lambda_n$ of $V$ (Proposition \ref{prop:ExistenceAlphaCartesianBasis}).  Write $ g = \sum_{i = 1}^n g_i \otimes \lambda_i$ with $g_1, \dots, g_n \in \Oc(X)$.

According to \cite[Lemme 3.1]{PoineauAngeliques} the inclusion $A \otimes_k V \subset A \otimes_k k'$ is isometric and Proposition \ref{Prop:CompatibiliyAlphaCartesianBasisTensorProduct} yields
$$ \| g\|_{A, k'} \ge \alpha^{-1} \max_{i = 1, \dots, n} \{ \| g_i\|_K |\lambda_i|\},$$
By hypothesis the point $x$ belongs to the holomorphically convex hull of $K$ in $X$, that is, $|h(x)| \le \| h\|_K$ for all holomorphic function $h$ on $X$. In particular,
$$
\sup_{\pr^{-1}(x)} |g| \le \max_{i = 1, \dots, n} \{ |g_i(x)| |\lambda_i|\} \le \max_{i = 1, \dots, n} \{ \| g_i \|_K |\lambda_i|\} \le \alpha \| g \|_{A, k'},
$$
which concludes the proof of the Lemma.
\end{proof}

The previous Lemma permits to prove the statement when $f \in \Oc(X) \otimes_k k'$. Indeed, applying it to every positive power of $f$,
$$ \sup_{\pr^{-1}(x)} |f| \le \lim_{n \to \infty} \alpha^{1/n} \| f^n \|^{1/n}_{A, k'} = \sup_{K'} |f|,$$
where the last equality comes from Proposition \ref{prop:SpectralNormExtensionScalars}.

Suppose $f \in \Oc(X')$. Let $\cal{D}$ be a cover of $X$ for the G-topology. For an affinoid domain $D$ of the family $\cal{D}$ denote by $\| \cdot \|_D$ the norm on the affinoid $k$-algebra $A = \Oc(D)$. The family of norms $\{ \| \cdot \|_D \}_{D \in \cal{D}}$ makes the $k$-vector space $\Oc(X)$ a $\cal{D}$-Banachoid space (see Definition \ref{def:MBanachoidSpace}). Similarly the family of sup norms $\{ \| \cdot \|_{D \times k'}\}_{D \in \cal{D}}$ on the $k'$-vector space $\Oc(X')$ makes the latter a $\cal{D}$-Banachoid space. The natural map $\Oc(X) \otimes_k k' \to \Oc(X')$ is injective and induces a bi-bounded isomorphism
$$ \Oc(X) \hotimes_k k' \stackrel{\sim}{\too} \Oc(X'),$$
(see Theorem \ref{thm:extensioncohomology}). Let $D_1, \dots, D_r \in \cal{D}$ be affinoid domains covering $K$. Then for every $\epsilon > 0$ there exists $g_\epsilon \in \Oc(X) \otimes k'$ such that $\| f - g_{\epsilon}\|_{D_i \times k'} < \epsilon$  for $i = 1, \dots, r$. In particular
$$ \| f - g_\epsilon \|_{K'} \le \max_{i = 1, \dots, n} \{\sup_{D_i \times k'} |f - g_\epsilon|\} < \epsilon. $$
 Then, according to the preceding case, there exists $y \in K$ such that
\begin{align*}
\| f \|_{\pr^{-1}(x)} &\le \| g_{\epsilon} \|_{\pr^{-1}(x)} + \epsilon \le \| g_{\epsilon} \|_{K'} + \epsilon \le \| f \|_{K'} + 2\epsilon,
\end{align*}
and we conclude letting $\epsilon$ tend to $0$.
\end{proof}

\subsection{Holomorphic convex hulls and products} For the sake of completeness, we sketch here the proof of the compatibility of the construction of the holomorphic convex hulls with taking products, namely the following statement:

\begin{proposition} \label{Prop:CompatibilityHolConvexhullsProducts} For $i = 1, 2$ let $X_i$ be a $k$-analytic space and $K_i$ be a compact subset of $X_i$. Assume that there exists $i\in \{1,2\}$ and a non-trivially valued analytic extension~$k'$ of~$k$ such that $X_{i}$ is separated and countable at infinity and that $H^1(X_{i}\times_{k}k',\Oc_{X_{i}\times_{k}k'})$ is finite-dimensional over~$k'$. Set $X := X_1 \times_k X_2$ and
$$ K:= \pr^{-1}_1(K_1) \cap \pr_2^{-1}(K_2),$$
where for $i = 1, 2$, $\pr_i \colon X \to X_i$ is the projection on the $i$-th factor. Then,
\begin{equation*} \pushQED{\qed} 
\hat{K}_X = \pr_1^{-1}(\hat{K}_{1, X_1}) \cap \pr_2^{-1}(\hat{K}_{2, X_2}). 
\qedhere
\popQED
\end{equation*}
\end{proposition}

The argument is analogous to the one of Proposition \ref{Prop:CompatibilityHolConvexhullsExtScalars} once one took care to replace the use Theorem \ref{thm:FXGY} by that of Theorem \ref{thm:extensioncohomology}, and Proposition \ref{prop:SpectralNormExtensionScalars} by the following statement, proved along the same lines:

\begin{proposition} \label{prop:SpectralNormProduct} For $i = 1, 2$ let $X_i$ be a $k$-analytic space, $K_i \subset X_i$ a compact subset and $A_i$ the ring of germs along $K$. Let $X := X_1 \times_k X_2$ and
$$ K := \pr^{-1}_1(K_1) \cap \pr^{-1}_2(K_2),$$
where, for $i =1, 2$, $\pr_i \colon X \to X_i$ is the projection of the $i$-th coordinate. Denote by $\| \cdot \|_{A_1, A_2}$ the tensor norm on $A_1 \otimes_k A_2$ and $\rho_{A_1,A_2}$ the associated spectral norm. Then, for every $f \in A_1 \otimes_k A_2$,
$$ \pushQED{\qed} 
\rho_{A_1, A_2}(f) := \lim_{n \to \infty} \| f \|_{A_1, A_2}^{1/n} = \sup_{K} |f|. 
\qedhere
\popQED$$
\qed
\end{proposition}

\subsection{Descent of holomorphic separability} Let us first recall the notion of peaked peaked (\emph{cf}. \cite[Definition~5.2.1]{Berkovich90}), or universally multiplicative (\emph{cf}. \cite[D\'efinition~3.2]{PoineauAngeliques}), analytic extension.

\begin{definition}\label{def:universallymultiplicative}
An analytic extension~$k'$ of~$k$ is said to be \emph{universally multiplicative} if, for each analytic extension~$K$ of~$k$, the tensor product norm on the algebra $K \hotimes_k k'$ is multiplicative.
\end{definition} 

For instance, if  $r$ is an element of $\R_{>0}^N$ whose coordinates are free over $|k^\times|$, then the extension~$k_{r}/k$ (\emph{cf}. Definition~\ref{def:Definitionk_r}) is universally multiplicative.

\begin{proposition} \label{prop:DescendingHolomorphicSeparability} Let $X$ be a $k$-analytic space that is countable at infinity. Let~$k'$ be a universally multiplicative analytic extension of~$k$. If~$X_{k'}$ is holomorphically separable, then so is~$X$.
\end{proposition}

\begin{proof} Let $\pr_{k'/k} \colon X_{k'} \to X$ be the natural projection morphism. Let $x_1, x_2 \in X$. By assumption, for $\ell=1,2$, the tensor product norm on $\khat(x_{\ell}) \hotimes_k k'$ is multiplicative, hence defines a point $\sigma(x_{\ell})$ in $\pr_{k'/k}^{-1}(x_{\ell}) \subseteq X_{k'}$. 

Since $X_{k'}$ is holomorphically separable by hypothesis, up to permuting $x_1$ and $x_2$, there is a function $f \in \Oc(X_{k'})$ such that $|f(\sigma(x_1))| < |f(\sigma(x_2))|$. By Theorem~\ref{thm:extensioncohomology}, $\Oc(X) \otimes k'$ is dense in $\Oc(X_{k'})$, hence the function~$f$ can be taken in $\Oc(X) \otimes k'$.

Let $V$ be a finite dimensional $k$-vector subspace of $k'$ such that $f$ lies in $\Oc(X) \otimes V$. Let $\beta\in(1,+\infty)$ such that $\beta |f(\sigma(x_1))| < |f(\sigma(x_2))|$. Then Proposition \ref{prop:ExistenceAlphaCartesianBasis} yields a $\beta$-cartesian basis $\lambda_1, \dots, \lambda_n$ of $V$. Write $f= \lambda_1 f_1 + \cdots + \lambda_n  f_n $ with $f_1, \dots, f_n \in \Oc(X)$.

For $\ell = 1, 2$, by definition of~$\sigma(x_{\ell})$, we have $\| f \|_{\khat(x_\ell) \hotimes k'} = |f(\sigma(x_\ell))|$, hence 
$$ \beta^{-1} \max_{i = 1, \dots, n} \{ |\lambda_i| |f_i(x_\ell)|\} \le |f(\sigma(x_\ell)) | \le \max_{i = 1, \dots, n} \{ |\lambda_i| |f_i(x_\ell)|\}$$
(according to \cite[Lemme 3.1]{PoineauAngeliques} the inclusion $\khat(x_\ell) \otimes V \subset \khat(x_\ell) \hotimes k'$ is isometric).  It follows that
\[  \max_{i = 1, \dots, n} \{ |\lambda_i| |f_i(x_1)|\} \le \beta |f(\sigma(x_1))| < |f(\sigma(x_2))| \le  \max_{i = 1, \dots, n} \{ |\lambda_i| |f_i(x_2)|\}, \]
hence there exists $i\in \{1,\dotsc,n\}$ such that $|f_{i}(x_1)| < |f_{i}(x_2)|$.
\end{proof}

\begin{remark}
Let $X$ be a $k$-analytic space and assume that~$k$ is algebraically closed. Then, \cite[Corollaire 3.14]{PoineauAngeliques} ensures that each point of $x$ is \emph{universal}, that is to say the extension $\khat(x)/k$ is universally multiplicative. 

Assume that $X$~is countable at infinity and let $k'$ be an arbitrary analytic extension of~$k$. Using virtually the same proof as above, one can show that, if~$X_{k'}$ is holomorphically separable, then so is~$X$.
\end{remark}

\section{Characterization when the boundary is not necessarily empty}\label{sec:boundary}

\begin{proposition} \label{Prop:ExhaustedByCompactSteinImpliesCohomStein} Let $X$ be a $k$-analytic space W-exhausted by Liu domains $\{ D_i\}_{i \in \N}$ and $F$ a coherent sheaf on $X$. Then,
\begin{enumerate}
\item $F(X)$ is dense in $F(D_i)$ for all $i \in \N$;
\item $\H^q(X, F) = 0$ for all $i \in \N$.
\end{enumerate}
\end{proposition}

\begin{proof} We reproduce Kiehl's argument. For all $i \in \N$, the set of global sections $F(D_i)$ forms a finitely generated $\Oc_X(D_i)$-module (Proposition \ref{Prop:GlobalSectionsCoherentSheafOnLiuSpaces} (4)). The norms on $F(D_i)$ obtained as quotient norms through surjections $ \Oc_X(D_i)^n \to F(D_i)$ are all equivalent. Recursively pick a norm $\| \cdot \|_i$ in this equivalence class normalized so that $ \| s\|_{i-1} \le \| s \|_{i}$ for all $s \in F(D_i)$.

(1) Let $i\in \N$, $s \in F(D_i)$ and $\epsilon > 0$. Set $t_0 = s$ and for $n \ge 1$ define recursively a sequence of sections $t_n \in F(D_{i+n})$ such that 
$$ \| t_{n} - t_{n - 1}\|_{i + n - 1} \le \epsilon / n.$$
For all $j \ge i$ the sequence $\{ t_{n}\}_{n \ge j - i}$ converges on $D_j$. We thus obtain a global section $t \in F(X)$. On the other hand, for all $n \ge 1$,
$$ \| t_n - s\|_{i} \le \max_{\ell = 0 , \dots, n-1} \| t_{\ell + 1} - t_\ell \|_{i} \le \epsilon,$$
thus $\| t - s\|_{i} \le \epsilon$.

(2) Since Liu spaces are cohomologically Stein (Corollary \ref{Cor:LiuSpaceCohomStein}), the cover for the $G$-topology $\cal{D} = \{ D_i\}_{i \in \N}$ is acyclic. Therefore the coherent cohomology on $X$ can be computed as the \v{C}ech cohomology with respect to $\cal{D}$: for all $q \ge 0$ and all coherent sheaf $F$ on $X$,
$$ \H^q(X, F) = \check{\H}^q(\cal{D}, F).$$

We argue by induction on $q \ge 1$. In order to show $\H^1(X, F) = 0$, it suffices to show that, given a sequence of sections $t_i \in F(D_i)$, there exists a sequence $s_i \in F(D_i)$ such that, for all $i \in \N$, $ t_i = s_i - s_{i +1}$. Using~(1), we may construct a sequence $s'_i \in F(D_i)$ such that, for all $i \in \N$, $ \| t_i - (s'_i - s'_{i+1})\|_{i} \le 2^{-i}$. For $i\in \N$, write $t_i' := t_i - (s'_i - s'_{i+1})$. For each $i\in \N$, the series  $ \sum_{n = 0}^\infty t'_{i +n}$
converges on~$D_i$ to an element $s_i'' \in F(D_i)$. For each $i\in \N$, we then have $ t_i' = s_i'' - s_{i+1}''$ and $ t_i = (s_i' + s_i'') - (s_{i+1}' + s_{i+1}'')$.

Let $q \ge 2$ and suppose that $(q-1)$-th cohomology group $\H^{q-1}(X, F)$ vanishes. For $i \in \N$, let $\cal{D}_i = \{  D_0, \dots, D_i \}$ and $\CC^\bullet_{i}$ be the \v{C}ech complex $\CC^\bullet(\cal{D}_i, F)$. The complexes $\CC^\bullet_i$ form a projective system through the maps $\phi_i \colon \CC^\bullet_{i+1} \to \CC^\bullet_i$ induced by projections. The projective limit
$$\CC^\bullet := \varprojlim_{i \in \N} \CC^\bullet_i, $$
is identified with the \v{C}ech complex $\CC^\bullet(\cal{D}, F)$. With these notations, for all $p \in \N$,
$$ \H^p(X, F) = \H^p(\CC^\bullet).$$
If the hypotheses Lemma \ref{lem:ProjectiveSystemCochainComplex} are satisfied for $q-1$, then 
$$ \H^q(X, F) = \H^q(\CC^\bullet) = \varprojlim_{i \in \N} \H^q(\CC^\bullet_i) = \varprojlim_{i \in \N} \H^q(D_i, F) = 0,$$
which concludes the proof.

For an integer $p \ge 0$ and $\ell = (\ell_0, \dots, \ell_p) \in \{ 0, \dots, i\}^{p+1}$,
$$ D_{\ell_0} \cap \cdots \cap D_{\ell_p} = D_{\min \ell},$$
where $\min \ell = \min \{ \ell_0, \dots, \ell_p\}$. In particular,
$$ \CC^p_{i} = \prod_{\ell \in \{ 0, \dots, i\}^{p+1}} F(D_{\min \ell}),$$
and the map $\phi_i^p \colon \CC_{i+1}^p \to \CC_{i}^p$ is surjective.  Let us also show that hypothesis (2) of Lemma \ref{lem:ProjectiveSystemCochainComplex} is satisfied. Since $\cal{D}_i$ is an acyclic cover, we have  $ \H^p(\CC^\bullet_{i}) = \H^p(D_i, F)$ for all $i \in \N$. In particular $\H^{q -1}(\CC^\bullet_{i})$ vanishes by inductive hypothesis.
\end{proof}

\begin{theorem} \label{Thm:EquivalenceBoundary} Let $X$ be a separated, countable at infinity $k$-analytic space. The following are equivalent:

\begin{enumerate}
\item for every analytic extension $k'$ of $k$ the $k'$-analytic space $X_{k'}$ is cohomologically Stein;
\item $X$ is holomorphically convex, $\Oc_X$ is universally acyclic and one of the following conditions is verified:
\begin{enumerate}
\item $X$ is holomorphically separable;
\item for every analytic extension $k'$ of $k$, $X_{k'}$ is rig-holomorphically separable;
\item there is a non-trivially valued analytic extension~$k'$ of $k$ such that $X_{k'}$~is strict and rig-holomorphically separable;
\end{enumerate}
\item $X$ is W-exhausted by Liu domains;
\item $\Oc_X$ is universally acyclic and, for every analytic extension $k'$ of $k$ and every coherent sheaf of ideals $I$ on $X_{k'}$ such that $\Oc_{X_{k'}} / I$ is supported at a discrete set of points, the cohomology group $\H^1(X_{k'}, I)$ vanishes.
\end{enumerate} 
Moreover, if $k$ is non-trivially valued, the preceding conditions are equivalent to $X$~being holomorphically convex and rig-holomorphically separable, and $\Oc_X$ being acyclic.
\end{theorem}

\begin{proof}[{Proof of Theorem \ref{Thm:EquivalenceBoundary}}] We first show the following implications:
$$
(3) \Longrightarrow (1) \Longrightarrow (4) \Longrightarrow (\textup{2b}) \Longrightarrow (3).
$$

(3) $\Rightarrow$ (1) The property of being $W$-exhausted by Liu domains is stable under scalar extension. Therefore one reduces to the case $k ' = k$ and applies Proposition~\ref{Prop:ExhaustedByCompactSteinImpliesCohomStein}.

\medskip

(1) $\Rightarrow$ (4) Clear.

\medskip
(4) $\Rightarrow$ (2b) Suppose that for every analytic extension $k'$ of $k$ and every coherent sheaf of ideals $I$ on $X_{k'}$ such that $\Oc_{X_{k'}} / I$ is supported at a discrete set of points, $$\H^1(X_{k'}, I) = 0.$$

The argument for the holomorphic separation is the classical one: let $k'$ be an analytic extension of $k$, $x, x'$ rigid points of $X' :=X_{k'}$ and $I$ the ideal sheaf defining the closed analytic subspace $\{ x \} \cup \{ x'\}$ in $X'$. By hypothesis the cohomology group $\H^1(X', I)$ vanishes, thus the sequence
$$ 0 \too I(X') \too \Oc_{X'}(X') \too \H^0(\{ x \} \cup \{ x'\}, \Oc_{X'}) = k(x) \times k(x') \too 0,$$
is exact. Therefore there exists $f \in \Oc(X')$ such that $f(x) = 0$ and $f(x') = 1$. 

For the holomorphic convexity, let $K$ be a compact subset of $X$. According to \cite[Corollaire 5.12]{PoineauAngeliques} it suffices to show that the holomorphically convex hull~$\hat{K}_X$ is sequentially compact. 

Let $S = \{ x_i \}_{i \in \N}$ be a sequence of points in $\hat{K}_X$. Let $k'$ be an analytic extension of~$k$ endowed for every $i \in \N$ with an isometric embedding $\epsilon_i \colon \khat(x_i) \to k'$. For $i \in \N$, let~$x'_i$ be the $k'$-rational point of $X' := X_{k'}$ associated to $x_i$. According to Proposition \ref{Prop:CompatibilityHolConvexhullsExtScalars} the points $x'_i$ belong to the holomorphically convex hull of $K' = \pr_{k'/k}^{-1}(K)$ in~$X'$, that is $ |f(x'_i)| \le \| f \|_{K'}$ for all $i \in \N$ and all $f \in \Oc(X')$.

Suppose by contradiction that the sequence $S$ is discrete in $X$. Then the sequence $S' = \{ x'_i \}_{i \in \N}$ is discrete too. Let $I$ be the coherent sheaf of $\Oc_{X'}$-ideals made of holomorphic functions vanishing identically on $S'$. The cohomology group $\H^1(X', I)$ vanishes by hypothesis, thus the short exact sequence of $\Oc_{X'}$-modules,
$$ 0 \too I \too  \Oc_{X'} \too \Oc_{X'}/I \too 0, $$
gives a short exact sequence of $k'$-vector spaces,
$$ 0 \too I(X') \too \Oc_{X'}(X') \stackrel{\pi}{\too} \H^0(X', \Oc_{X'}/I) \too 0, $$
where  $ \pi(f) = (f(x'_i))_{i \in \N}$. In particular there exists a holomorphic function $f$ on~$X'$ such that  $|f(x'_i)| \ge i$ for all $i \in \N$. This implies $\| f \|_{K'} = +\infty$ contradicting the compactness of $K'$.

\medskip 

(2b) $\Rightarrow$ (3) We need here the following enveloping argument, which we state and prove separately for later use:

\begin{lemma} \label{lemma:AffinoidDomainsEnvelopingCompacts} Let $X$ be a  holomorphically convex $k$-analytic space and $K$ a compact subset of~$X$. 

Then, there exist a non-negative integer $n$, positive real numbers $r_1, \dots, r_n$, a morphism of $k$-analytic spaces $f \colon X \to \A^{n, \an}_k$ and open neighbourhoods $U$ of $\hat{K}_X$ and~$V$ of~$f(U)$ such that:
\begin{enumerate}
\item the induced morphism $f_{\rvert U} : U \to V$ is topologically proper;
\item the closed disc $\D := \D(r_1, \dots, r_n)$ of radii $r_1, \dots, r_n$ is contained in $V$;
\item the holomorphically convex hull $\hat{K}_X$ is contained in the interior of  $f^{-1}(\D) \cap U$ in $X$.
\end{enumerate}
\end{lemma}

\begin{proof} Since $X$ is holomorphically convex, $\hat{K}_X$ is compact. Thus one may suppose $K = \hat{K}_X$. For every point $x \in X \smallsetminus K$ there exists a holomorphic function $f \in \Oc(X)$ such that $ \| f \|_K < |f(x)|$.

Let $W$ be a relatively compact open neighbourhood of $K$. The topological border $\partial W$ of $W$ in $X$ is compact and does not meet $K$, thus there exist finitely many holomorphic functions $f_1, \dots, f_n \in \Oc(X)$ such that
$$ \max \left\{ \frac{|f_1(x)|}{\| f_1 \|_K}, \dots, \frac{|f_n(x)|}{\| f_n \|_K} \right\} > 1,$$
for all $x \in \partial W$. Let $f = (f_1, \dots, f_n) \colon X \to \A_k^{n, \an}$ be the induced morphism and, for $i = 1, \dots, n$, $r_i := \| f_i \|_K$. The image of a point $x \in X$ belongs to the disc $\D := \D(r_1, \dots, r_n)$ if and only if $|f_i(x)| \le \| f_i \|_K$ for all $i = 1, \dots, n$. In particular, $K$ is contained in $f^{-1}(\D)$ and $\partial W \cap f^{-1}(\D) = \emptyset$.

Set $V := \A^{n, \an}_k \smallsetminus f(\partial W)$. Since $\partial W$ is compact and $f(\partial W)$ does not meet $\D$, $V$~is an open neighbourhood of~$\D$. The subset $U := W \smallsetminus f^{-1}(f(\partial W))$ is open in~$X$ and the map $f_{\rvert U} \colon U \to V$ is topologically proper.\footnote{If $f \colon X \to Y$ is a continuous map between locally compact topological spaces and $W$ is a relatively compact open subset of $X$, then the induced map $W \smallsetminus f^{-1}(f(\partial W)) \to Y \smallsetminus f(\partial W)$ is topologically proper.} 
\end{proof}

Suppose that~$X$ is holomorphically convex. Since $X$~is countable at infinity there is an increasing sequence of compact subsets $\{ K_i \}_{i \ge 0}$ such that $K_i$ is contained in the interior of $K_{i + 1}$ and which cover $X$.

\begin{claim}\label{claim:Di} There is an non-decreasing sequence $\{ D_i \}_{i \ge 0}$ of compact analytic domains of~$X$ such that, for all $i \ge 0$, $D_i$ contains $K_i$ and there exist functions $f_{i1}, \dots, f_{i n_i} \in \Oc(X)$, positive real numbers $r_{i 1}, \dots, r_{in_i}$ and an open subset $U_i$ of~$X$ such that
$$ D_i := \{ x \in U_i : |f_{ij}(x)| \le r_{ij} \textup{ for all } j = 1, \dots, n_i\}.$$

Moreover, we may choose the $r_{ij}$'s in any dense sub-$\Q$-vector space of~$\R_{>0}$.
\end{claim}

\begin{proof}[Proof of the Claim] In order to define the sequence, it is convenient to start the induction at $-1$ and set $D_{-1} = \emptyset$. Then, for all $i \ge 0$ and supposing $D_{i -1}$ defined, apply Lemma \ref{lemma:AffinoidDomainsEnvelopingCompacts}
 with $K = K_i \cup D_{i - 1}$: let $f_i \colon X \to \A_k^{n_i, \an}$, $D$, $U$, $V$ be respectively the morphism, the disc, the open neighbourhood of $\hat{K}_X$ and the neighbourhood of~$D$ given by the lemma. If~$R$ is a dense sub-$\Q$-vector space of~$\R_{>0}$, then we may take $D$ to have radii in~$R$ while the conclusions of Lemma \ref{lemma:AffinoidDomainsEnvelopingCompacts} are still fulfilled. The analytic domain $D_{i} := f_i^{-1}(D) \cap U$ satisfies the required properties.
\end{proof}

Suppose, moreover, that $\Oc_{X}$ is universally acyclic and that, for every analytic extension $k'$ of $k$,  $X_{k'}$ is rig-holomorphically separable. Then, the sequence $\{ D_i \}_{i \ge 0}$ is a W-exhaustion of~$X$ by Liu domains. 

Indeed, for $i\ge 0$, $D_{i}$ is a union of connected components of 
$$ D'_i = \{ x \in X : |f_{ij}(x)| \le r_{ij} \textup{ for all } j = 1, \dots, n_i\} $$
and $\Oc_{D'_{i}}$ is universally acyclic by Proposition~\ref{prop:FunctionsOnRationalDomains}. 
Let~$\ell_{i}$ be an analytic extension of~$k$ that is universally multiplicative (\emph{cf}. Definition~\ref{def:universallymultiplicative}) and such that~$D_{i,\ell_{i}}$ is strict. This can always be achieved by choosing~$\ell_{i}$ of the form~$k_{r}$ for some $r \in \R_{>0}^N$ whose coordinates are free over~$|k^\times|$ (\emph{cf}. Definition~\ref{def:Definitionk_r}). By construction, $D_{i, \ell_{i}}$ is strict, compact, separated and $\Oc_{D_{i, \ell_{i}}}$ is universally acyclic. By hypothesis $D_{i, \ell_{i}}$ is rig-holomorphically separable and therefore Corollary \ref{Cor:StrictLiuSpacesAreHolSeparable} implies that $D_{i, \ell_{i}}$ is holomorphically separable. Proposition \ref{prop:DescendingHolomorphicSeparability} yields in turn that~$D_{i}$ is holomorphically separable, hence a Liu space.

Similarly, $D_{i}$~is a union of connected components of 
$$ D''_i = \{ x \in D_{i+1} : |f_{ij}(x)| \le r_{ij} \textup{ for all } j = 1, \dots, n_i\}, $$
and the restriction map $\Oc(D_{i+1}) \to \Oc(D''_i)$ has dense image by Proposition~\ref{prop:FunctionsOnRationalDomains}.

\medskip

In order to finish the proof of the theorem, it suffices to prove the following implications:
$$ (3) \Longrightarrow (\textup{2a}) \Longrightarrow (\textup{2b}) \Longrightarrow (\textup{2c}) \Longrightarrow (1).$$

(3) $\Rightarrow$ (2a) \emph{Holomorphic separation}. Let $\cal{D} = \{ D_i\}_{i \in \N}$ be a G-cover of $X$ by Liu domains such that the restriction map $\Oc(D_{i+1}) \to \Oc(D_i)$ has dense image. Two distincts $x, x' \in X$ belong to $D_i \subset X$ for $i$ big enough. Since $\Oc(X)$ is dense in $\Oc(D_i)$ it suffices to find $f \in \Oc(D_i)$ such that $|f(x)| \neq |f(x')|$. This exists as $D_i$ is by definition holomorphically separated.

\emph{Holomorphic convexity}. A compact subset $K \subset X$ is contained in $D_i$ for $i$ big enough. By density of $\Oc(X)$ in $\Oc(D_i)$, if a point $x \in X$ verifies $|f(x)| \le \| f\|_K$ for all $f \in \Oc(X)$, then it verifies the same inequality for all $f \in \Oc(D_i)$. By hypothesis $ \| f\|_K \le \| f\|_{D_i}$ for all holomorphic function $f \in \Oc(D_i)$, thus
$$ \hat{K}_X \subset \cal{M}(\Oc(D_i)) = D_i,$$
as subsets of $X$ (Corollary \ref{cor:SpectrumLiuSpace}). In particular $\hat{K}_X$ is compact.

\medskip 

(2a) $\Rightarrow$ (2b) According to Proposition \ref{Prop:SSpaceStableScalarExt}, being a Liu space is stable under extending scalars. Thus one reduces to the case $k ' = k$ where the result is Lemma~\ref{lem:RigHolSepImpliesHolSep}.

\medskip

(2b) $\Rightarrow$ (2c) Clear.

\medskip

(2c) $\Rightarrow$ (1) Suppose that~$X$ is holomorphically convex, that $\Oc_{X}$ is universally acyclic and that there exists a non-trivially valued analytic extension~$k'$ of~$k$ such that~$X_{k'}$ is strict and rig-holomorphically separable. 

Let us consider a sequence $\{ D_i \}_{i \ge 0}$ of analytic domains of~$X$ as in Claim~\ref{claim:Di} with all the $r_{i,j}$'s in $\sqrt{|k'^\times|}$. Then, for each $i\ge 0$, $D_{i,k'}$ is strict and it follows from Proposition~\ref{prop:FunctionsOnRationalDomains} and Corollary \ref{Cor:StrictLiuSpacesAreHolSeparable} that $\{ D_{i,k'} \}_{i \ge 0}$ is a W-exhaustion of~$X_{k'}$ by Liu domains (by the same arguments as in the last part of the proof of (2b) $\Rightarrow$ (3)). From (3) $\Rightarrow$ (1), we know that, for every analytic extension~$k''$ of~$k'$, $X_{k''}$ is cohomologically Stein. The result now follows from Theorem~\ref{thm:extensioncohomology}.
\end{proof}

\section{Characterization when the boundary is empty} \label{sec:ProofOfMainTheorem}

\begin{theorem}[Remmert reduction theorem]\label{thm:Remmert} Let $X$ be a $k$-analytic space without boundary, holomorphically convex and countable at infinity. 

Then there exists a $k$-analytic space $S$, the \emph{Remmert factorization} of $X$, without boundary and W-exhausted by affinoid domains, together with a morphism $\pi \colon X \to S$ satisfying the following properties:
\begin{enumerate}
\item $\pi$ is proper, surjective and with geometrically connected fibers;
\item $\pi^\sharp \colon \Oc_S \to \pi_\ast \Oc_X$ is an isomorphism;
\item given a $k$-analytic space $S'$ and holomorphic map $f \colon X \to Y$ constant on the fibers of~$\pi$, there exists a unique holomorphic map $\tilde{f} \colon S \to Y$ such that $f = \tilde{f} \circ \pi$.
\end{enumerate}
In particular, given a holomorphically separable $k$-analytic space~$Y$ and a morphism $f \colon X \to Y$, there exists a unique morphism $\tilde{f} \colon S \to Y$ such that $f = \tilde{f} \circ \pi$.
\end{theorem}

\begin{remark} What might a putative Remmert reduction theorem be when the boundary is no longer empty? According to Theorem \ref{Thm:EquivalenceBoundary} it may seem reasonable to ask whether there exists, for a $k$-analytic space $X$ countable at infinity, holomorphically convex and such that $\Oc_X$ is universally acylic, a $k$-analytic space $W$-exhausted by Liu domains $S$ and a proper holomorphic map $\pi \colon X \to S$ satisfying the same formal properties as the Remmert factorization. The following example shows that the properness of the map $\pi$ has to be dropped  from the requirements.  

Suppose $k$ of characteristic $0$ and let $A$ be  an abelian scheme over the ring of integers of $k$. Let $\EE(A)$ be the universal vector extension of $A$ and consider Raynaud's generic fibre $\EE(A)_\eta$ of the formal scheme obtained by taking the formal completion of $\EE(A)$ along its special fibre. 

The $k$-analytic space $X= \EE(A)_\eta$ is compact and has non empty boundary: it is a disc bundle over the abelian variety $A^\an$. In a forthcoming paper we will show that every function on $X$ is constant and $\Oc_X$ is acyclic. Then every holomorphic map $X \to S$ with target a holomorphically separable space is constant: in this case, the Remmert factorization of $X$ is simply $X \to \M(k)$.
 \end{remark}

\begin{proof} Let $\{ K_i \}_{i \ge 0}$ be an exhaustion of $X$ by compact subsets such that, for each~$i$, $K_{i}$ is the closure of its interior and $K_i$ is contained in the interior of $K_{i + 1}$.  

Set $n_{-1} = 0$ and $X_{-1}$, $U_{-1}$, $V_{-1}$ to be the empty set. For $i \ge 0$, we construct inductively an integer $n_i \ge 0$, a relatively compact open subset $U_i$ of $X$, a compact analytic domain~$X_i$ of $U_i$, an open subset $V_i$ of $\A^{n_i, \an}_k$ and holomorphic functions $f_{n_{i-1} + 1},  \dots, f_{n_i}$ on~$X$ such that:
\begin{enumerate}
\item the map $F_i = (f_1, \dots, f_{n_i}) \colon U_i \to V_i$  is proper;
\item $\D_i := \D(\| f_1\|_{K_i \cup \overline{U}_{i-1} }, \dots, \| f_{n_i}\|_{K_i \cup \overline{U}_{i-1}})$ is contained in $V_i$;
\item $X_i := F_i^{-1}(\D_i) \cap U_i$ contains the holomorphically convex hull of the compact subset $K_i \cup \overline{U}_{i-1}$.
\end{enumerate}

Set $K'_i := K_i \cup \overline{U}_{i - 1}$ and consider a relatively compact open subset $W_i$ of $X$ containing the holomorphically convex hull of $K'_i$ in $X$ (which is compact since~$X$ is holomorphically convex). Since $\partial W_i$ is compact, there are holomorphic functions $f_{n_{i -1} + 1}, \dots, f_{n_{i}}$ on $X$ such that, for all $x \in \partial W_i$,
$$\max_{j = 1, \dots, n_i} \frac{|f_j(x)|}{\| f_j \|_{K'_i}} > 1.$$
Set $V_i := \A^{n_i, \an}_k \smallsetminus F_i(\partial W_i)$ and $U_i = W_i \smallsetminus F_i^{-1}(F_i(\partial W_i))$. As in the proof of Lemma~\ref{lemma:AffinoidDomainsEnvelopingCompacts}, the induced map $F_i \colon U_i \to V_i$ is topologically proper. Since~$X$ is without boundary, so is~$U_{i}$, hence $F_{i}$~is proper.

The map $F_i \colon X_i \to \D_i$ being proper, by the ``Stein factorization theorem'' (\cite[Proposition 3.3.7]{Berkovich90}), there exists a $k$-analytic space $S_i$ together with morphisms $\pi_i \colon X_i \to S_i$, $\tilde{F}_i \colon S_i \to \D_i$ such that
\begin{enumerate}
\item $\pi_i$ is proper and surjective, $\tilde{F}_i$ is finite and $F_i = \tilde{F}_i \circ \pi_i$;
\item $\Oc_{S_i} \to \pi_{i \ast } \Oc_{X_i}$ is an isomorphism and $\pi_i$ has connected fibres;
\item given a $k$-analytic space $T$ and a map $g \colon X_i \to T$ constant on the fibers of~$\pi_i$, there exists a unique map $\tilde{g} \colon S_i \to T$ such that $g_i = \tilde{g}_i \circ \pi_i$.
\end{enumerate}
In particular $S_i$ is affinoid with $k$-affinoid algebra $\Oc(X_i)$. The restriction map $\Oc(X_{i+1}) \to \Oc(X_{i})$ is a bounded homomorphism of $k$-affinoid algebras and it induces a morphism of $k$-affinoid spaces $\epsilon_i \colon S_i \to S_{i+1}$. 

\begin{claim} The map $\epsilon_i$ identifies $S_i$ with a Weierstrass domain of $S_{i+1}$. \end{claim}

\begin{proof}[Proof of the Claim] Consider the following disc in $\A^{n_{i+1}, \an}_k$:
$$ \D'_i := \D_i \times \D(\| f_{n_i + 1}\|_{K_{i+1}'}, \dots, \| f_{n_{i + 1}}\|_{K_{i+1}'} ).$$

With this notation we have $ X_i = U_{i} \cap F_{i+1}^{-1}(\D'_i)$: indeed $U_{i} \cap F_{i+1}^{-1}(\D'_i)$  is contained in $X_i$ by definition; on the other hand, $\| f_j\|_{X_i} \le \| f_j\|_{K'_{i+1}}$ for all $j$, whence the converse inclusion.

Consider the analytic domain $Y := U_{i+1} \cap F_{i+1}^{-1}(\D'_i) $ in $X_{i+1}$. Let $T = \M(\Oc(Y))$ be the Stein factorization of the proper morphism $F_{i + 1} \colon Y \to \D'_i$. Since $\D'_i$ is an affinoid domain in $\D_{i+1}$, it follows that $T$ is an affinoid domain in $S_{i+1}$; moreover, the affinoid domain $T$ is a Weierstrass domain, as $\D'_{i}$ is a Weierstrass domain in $\D_{i+1}$. More explicitly, the map $F_{i+1} \colon X_{i+1} \to \D_{i+1}$ is proper, thus the sheaf  $E := F_{i+1 \ast} \Oc_{X_{i+1}}$ is coherent on $\D_{i+1}$. The disc $\D_i'$ is the Weierstrass domain in $\D_{i+1}$ given by the inequalities $|t_j| \le \| f_j\|_{X_i}$ for $j = 1, \dots, n_i$. It follows that the global sections of $E$ on $\D_i'$, $ E(\D_i') = \Oc(Y)$,  can be identified with the affinoid algebra
$$ E(\D_{i+1}) \left\{ \frac{f_1}{\| f_1\|_{X_i}}, \dots, \frac{f_{n_i}}{\| f_{n_i}\|_{X_i}}\right\} = \Oc(X_{i+1}) \left\{ \frac{f_1}{\| f_1\|_{X_i}}, \dots, \frac{f_{n_i}}{\| f_{n_i}\|_{X_i}}\right\}. $$
In particular, $T = \M(\Oc(Y))$ is the Weierstrass domain in $S_{i + 1}$ given by the inequalities $|t_j| \le \| f_j\|_{X_i}$ for $j = 1, \dots, n_i$.

In order to conclude that $S_i$ is a Weierstrass domain in $S_{i + 1}$, remark that $X_i$ is contained in $Y$ because $U_i$ is contained in $U_{i+1}$. Furthermore, $X_i$ is clopen in $Y$: it is open because of the equality $X_i = Y \cap U_i$ and closed because of its compactness. Therefore $X_i$ is a finite union of connected of components of $Y$. By writing $Y = Y' \sqcup X_i$, it follows that $\Oc(Y)$ is isomorphic to the product ring $\Oc(X_i) \times \Oc(Y')$. By passing to the Banach spectrum of these affinoid algebras, one has
$$ T = \M(\Oc(Y)) = \M(\Oc(X_i)) \sqcup \M(\Oc(Y')) = S_i \sqcup \M(\Oc(Y')),$$
that is, $S_i$ is a union of connected components of~$T$. In particular $S_i$ is a Weierstrass domain of $S_{i+1}$.
\end{proof}

The $k$-analytic space $S = \bigcup_{i \in \N} S_i$ is by definition W-exhausted by affinoid domains. The map $\pi \colon X \to S$ is proper as $\pi_i \colon X_i \to S_i$ is proper for every $i \ge 1$. Since the space $X$ is without boundary, a given analytic domain $X_i$ is contained in the interior of some analytic domain $X_j$ for $j \ge i$ big enough. By properness of $\pi$, it follows that $S_i$ is contained in the interior of $S_j$ and, in particular, $S$ is without boundary. The properties of the map $\pi$ are deduced from the ones of $\pi_i$ as $\pi^{-1}(S_i)  = X_i$ for all $i \ge 0$.
\end{proof}

\begin{theorem} \label{Thm:EquivalenceWithoutBoundary} Let $X$ be a $k$-analytic space without boundary and countable at infinity. The following are equivalent:

\begin{enumerate}
\item for every analytic extension $k'$ of $k$, the $k'$-analytic space $X_{k'}$ is cohomologically Stein;
\item $X$ is holomorphically convex and one of the following conditions is verified:
\begin{enumerate}
\item $X$ is holomorphically separable;
\item for every analytic extension $k'$ of $k$, $X_{k'}$ is rig-holomorphically separable;
\item there is a non-trivially valued analytic extension~$k'$ of $k$ such that $X_{k'}$ is rig-holomorphically separable;
\end{enumerate}
\item $X$ is W-exhausted by affinoid domains;
\item for every analytic extension $k'$ of $k$ and every coherent sheaf of ideals $I$ on $X_{k'}$ such that $\Oc_{X_{k'}} / I$ is supported at a discrete set of points, the cohomology group $\H^1(X_{k'}, I)$ vanishes.

\end{enumerate} 
Moreover, if $k$ is non-trivially valued, the preceding conditions are equivalent to $X$~being holomorphically convex and rig-holomorphically separable.
\end{theorem}

\begin{proof} Implications (1) $\Rightarrow$ (4) and (2b) $\Rightarrow$ (2c) are clear.

\medskip

(3) $\Rightarrow$ (2a) \emph{Holomorphic separability}. Let $x, x' \in X$ be distinct points and let $D \subset X$ be a Weierstrass domain containing them. Since $\Oc(X)$ is dense in $\Oc(D)$ it suffices to find $f \in \Oc(D)$ such that $|f(x)| \neq |f(x')|$. This clearly exists as $D$ coincides with $\M(\Oc(D))$.

\emph{Holomorphic convexity}. Let $K \subset X$ be a compact subset and $D \subset X$ be a Weierstrass domain containing it. By density of $\Oc(X)$ in $\Oc(D)$, if a point $x \in X$ verifies $|f(x)| \le \| f\|_K$ for all $f \in \Oc(X)$, then it verifies the same inequality for all $f \in \Oc(D)$. By hypothesis $ \| f\|_K \le \| f\|_D$ for all holomorphic function $f \in \Oc(D)$, thus $ \hat{K}_X \subset \cal{M}(\Oc(D)) = D$
as subsets of $X$. In particular $\hat{K}_X$ is compact.

\medskip 

(3) $\Rightarrow$ (2b) Statement (3) is stable under extension of scalars and we already prove that it implies holomorphic separability. We conclude because the property of holomorphic separability implies that the rigid points are separable by holomorphic functions (Lemma \ref{lem:RigHolSepImpliesHolSep}).

\medskip

(3) $\Rightarrow$ (1) Being W-exhausted by affinoid domains is stable under extension of scalars. Therefore it suffices to prove for the result for $k' = k$. In this case it is Theorem \ref{Thm:SpacesExhaustedByWeierstrassAreStein}.

\medskip 

The proof of (4) $\Rightarrow$ (2b) is exactly as (4) $\Rightarrow$ (2b) in the proof of Theorem \ref{Thm:EquivalenceBoundary}. 

\medskip

For the proof of (2a) $\Rightarrow$ (3) and (2c) $\Rightarrow$ (3) consider the Remmert reduction $S$ of $X$, which is possible because $X$ is holomorphically convex and without boundary by hypothesis. In order to prove that $X$ is W-exhausted by affinoid domains it suffices to show that the canonical map $\pi \colon X \to S$ is finite. Since it is already proper, it is the matter of proving that the fibers of $\pi$ are finite. 

Assuming (2c), it is sufficient to prove that $\pi_{k'} \colon X_{k'} \to S_{k'}$ is finite. Thanks to \cite[Proposition 3.3.2]{Berkovich90} and \cite[9.6.3 Corollary 6]{BGR}, it suffices to prove that, for every $y \in X_{k', \rig}$ the fiber $\pi_{k'}^{-1}(y)$ is finite (the $k'$-analytic spaces $X_{k'}$, $S_{k'}$ are strict). Since $\pi_{k'}$ is proper, for a rigid point $y \in S_{k'}$ the fiber $Z := \pi_{k'}^{-1}(y)$ is a closed analytic subspace of $X_{k'}$ which is proper over $k$. Suppose $Z$ is of positive dimension and take two distinct points $z, z' \in Z$: by hypothesis there is a holomorphic function $f$ on $X_{k'}$ such that $f(z) = 0$ and $f(z') = 1$. In particular $f_{\rvert Z}$ is non constant, contradicting the properness of $Z$.

Assuming (2a), the argument is similar. For a point $s \in S$ the fiber $Z := \pi^{-1}(s)$ is naturally endowed with the structure of a $\khat(s)$-analytic space, strict and proper over $\khat(s)$. If it is of positive dimension, pick two distinct points $z, z' \in Z$. By hypothesis there is a function $f$ such that $|f(z)| \neq |f(z')|$. In particular the restriction of~$f$ to~$Z$ cannot be constant, contradicting the properness of $Z$.
\end{proof}

\appendix

\section{Banachoid spaces} \label{sec:Banachoid}

By definition, the ring of global sections of an affinoid space is a Banach algebra. From the algebraic point of view, this is very convenient since the theory of Banach algebras is well documented and many results are available.

However in this article we are interested in spaces with no boundary, and their rings of global sections are no longer Banach algebras but only Fr\'echet algebras at best. For similar reasons related to de Rham cohomology, Andrea Pulita and the second named author have developed a theory of normoid Fr\'echet spaces (\textit{i.e.} Fr\'echet spaces with a distinguished family of seminorms as part of the data) in~\cite{PoineauPulitaBanachoid}, inspired by Gruson's work~\cite{Gruson} in the setting of Banach spaces. We recall here the basic definitions and results for the convenience of the reader.

\begin{definition}
Let~$M$ be a non-empty set. An \emph{$M$-normoid space} is a $k$-vector space~$U$ endowed with a family of seminorms $u = (u_{m})_{m\in M}$.
\end{definition} 

We endow~$U$ with the uniform structure and the topology induced by~$u$. In more concrete terms, this means for instance that a sequence $(x_{n})_{n\ge 0}$ in~$U$ tends to~0 if, and only if, for each $m\in M$, the sequence $(u_{m}(x_{n}))_{n\ge 0}$ tends to~0. 

\begin{definition} \label{def:MBanachoidSpace}
Let~$M$ be a non-empty set. An \emph{$M$-Banachoid space} is an $M$-normoid space that is Hausdorff and complete.
\end{definition}

Obviously, any Banach space $(A,\nm)$ gives rise to a Banachoid space $(A,u_{\nm})$, where~$u_{\nm}$ is the family containing the single element~$\nm$.

The main example we have in mind in the following. Let~$X$ be a $k$-analytic space and let~$\cal{D}$ be an affinoid covering of~$X$ for the $G$-topology. For each $D \in \cal{D}$, denote by~$u_{D}$ the seminorm obtained by composing the norm on the $k$-affinoid algebra~$A_{D}$ associated to~$D$ with the restriction map $\Oc(X)\to A_{D}$. The space~$\Oc(X)$ endowed with $u=(u_{D})_{D\in \cal{D}}$ is a Banachoid space.

If~$\Fc$ is a coherent sheaf on~$X$, a similar construction can be used to put a Banachoid structure on the space~$F(X)$, by first endowing each $F(D)$ with the norm coming from a surjection $\Oc(D)^{n_{D}} \to F(D)$. 

Note that, for each complete valued extension~$k'$ of~$k$, the set $\{D_{k'} \mid D \in \cal{D}\}$ is an affinoid covering of~$X_{k'}$ for the $G$-topology, hence we get an induced Banachoid structure on $F(X_{k'})$.

We refer to \cite[Definition~1.9]{PoineauPulitaBanachoid} for more details about those constructions.

\begin{definition} 
Let $(U,(u_{m})_{m\in M})$ and~$(V,(v_{n})_{n\in N})$ be Banachoid spaces. We say that a $k$-linear map $f : U \to V$ is \emph{bounded} if, for each $n\in N$, there exists a real number~$C_{n}$ and a finite subset~$M_{n}$ of~$M$ such that 
\[\forall x \in U,\ v_{n}(f(x)) \le C_{n} \max_{m\in M_{n}} \{u_{m}(x)\}.\]
\end{definition}

In the situation of the examples above, it is not difficult to check that, if one changes the affinoid covering, the identity map between the two Banachoid spaces is bounded (see \cite[Lemma~1.10]{PoineauPulitaBanachoid}). Similarly, a morphism of coherent sheaves $\Fc \to \Gc$ gives rise to a bounded morphism of Banachoid spaces $F(X) \to G(X)$ (see \cite[Lemma~1.19]{PoineauPulitaBanachoid}).

\medbreak

We now come to the definition of completed tensor product. Let $(U,(u_{m})_{m\in M})$ and~$(V,(v_{n})_{n\in N})$ be Banachoid spaces. For each~$m\in M$,  $n\in N$ and $z \in U \otimes_{k} V$, set 
\begin{equation*}
u_{m}\otimes v_{n}(z) \;:=\; 
\inf\Bigl\{\max_{1\le i \le r} \{u_{m}(x_{i}) \cdot v_{n}(y_{i})\}\;\textrm{ such that } z = \sum_{i=1}^r x_{i}\otimes y_{i}\Bigr\}.
\end{equation*}
The map $u_{m} \otimes v_{n}$ is a seminorm on~$U\otimes_{k} V$. We denote by $U \hotimes_{k} V$ the Hausdorff completion of~$U\otimes_{k} V$. It is naturally a Banachoid space.

As one can expect, it is also possible to define a notion of bounded bilinear map and the space $U \hotimes_{k} V$ can then be proven to satisfy the expected universal property (see \cite[Proposition~3.1]{PoineauPulitaBanachoid}).

\medbreak

In the following, we will be interested specifically in the case where we tensor by a complete valued extension~$k'$ of~$k$ (seen as a Banachoid space with a single norm). Unlike in the usual case, exact sequences of Banachoid spaces may fail to remain exact after completed tensor product. To fix this, one needs to consider exact sequences of strict maps, where the image and coimage inherit the same Banachoid structure. It is very useful to know conditions where strictness is automatic. For Banach or Fr\'echet spaces over non-trivially valued fields, this is the case for surjective maps, and even maps with finite-dimensional cokernels, by Banach's open mapping theorem. To be able to use those results in the setting of Banachoid spaces, we introduce the following definition (see \cite[Definition~3.13]{PoineauPulitaBanachoid}).

\begin{definition} \label{def:FrechedNormoidSpace}
Let~$M$ be a non-empty set. An \emph{$M$-normoid Fr\'echet space} is an $M$-Banachoid space that admits a bi-bounded isomorphism into an $N$-Banachoid space with~$N$ countable.
\end{definition}

In the examples above, the spaces $\Oc(X)$ and $F(X)$ are normoid Fr\'echet if the covering~$\cal{D}$ is countable. Note that we can always choose a cover satisfying this property if~$X$ is paracompact (for instance if it is a curve or the analytification of an algebraic variety) and connected. 

By applying the completed tensor product to a suitable \v Cech complex, we obtain the following result (see \cite[Corollaire~3.21]{PoineauPulitaBanachoid}). Note that using the  \v Cech complex (for separated spaces) also enables to endow the spaces $\H^q(X,\Fc)$ for $q\ge 1$ with normoid Fr\'echet structures.

\begin{theorem}\label{thm:extensioncohomology}
Let~$X$ be a $k$-analytic space that is countable at infinity, let~$\Fc$ be a coherent sheaf on~$X$ and let~$k'$ be a complete valued extension of~$k$. Then, we have a canonical bi-bounded isomorphism
\begin{equation*}
F(X) \hotimes_{k} k' \stackrel{\sim}{\too} F_{k'}(X_{k'}).
\end{equation*}

Let~$q\ge 1$. Assume that~$X$ is separated and that there exists $\ell\in \{k,k'\}$ such that~$\ell$ is non-trivially valued and $\H^q(X_{\ell},\Fc_{\ell})$ is finite-dimensional. Then, we have a canonical bi-bounded isomorphism
\begin{equation*}
\H^q(X,\Fc) \hotimes_{k} k' \stackrel{\sim}{\too} \H^q(X_{k'},\Fc_{k'}).
\end{equation*}
\end{theorem}

We can also obtain results for global sections of sheaves on product of varieties (see \cite[Corollary~3.24]{PoineauPulitaBanachoid}).

\begin{theorem}\label{thm:FXGY}
Let~$X$ and~$Y$ be $k$-analytic spaces that are separated and countable at infinity. Assume that~$X$ or~$Y$ is finite-dimensional. Denote by~$\pr_{X}$ and~$\pr_{Y}$ the canonical projections from~$X\times_{k} Y$ to~$X$ and~$Y$ respectively. 

If~$\Fc$ and~$\Gc$ are universally acyclic coherent sheaves on~$X$ and~$Y$ respectively, then the coherent sheaf $\pr_{X}^\ast\Fc \otimes \pr_{Y}^\ast\Gc$ on~$X\times_{k} Y$ is universally acyclic too. 
\end{theorem}

\section{Zariski-trivial analytic spaces}\label{sec:Zariski-trivial}

Let $k$ be a complete non-archimedean field. Recall the notation~$k_{r}$ from Definition~\ref{def:Definitionk_r}. The aim of this section is to prove the following result:

\begin{proposition} \label{Prop:ZariskiTrivialSpaces}Let $X$ be a non-empty $k$-analytic space whose Zariski topology is the trivial topology. If $k$ is non-trivially valued, there are real numbers $r_1, \dots, r_n$ free over $|k^\times|$ such that and a finite local $k_r$-algebra $A$ (in particular, Artinian) such that $X = \M(A)$.
\end{proposition}

\begin{lemma} Let $X$ be a $k$-analytic space. The subset of $X$ made of points $x \in X$ having an affinoid neighbourhood, is open and dense.
\end{lemma}

\begin{proof} The openness is clear. For the density, the $k$-analytic space $X$ may be assumed non-empty and Hausdorff. In particular the affinoid domains of $X$ are closed. It suffices to prove that there is a point having an affinoid neighbourhood.

Let $x \in X$ and $D_1, \dots, D_n$ be affinoid domains such that $D = D_1 \cup \cdots \cup D_n$ is a neighbourhood of $x$. Suppose the family $D_1, \dots, D_n$ minimal for this property. Let $U$ be an open neighbourhood of $x$ contained in $V$. By minimality, $U$ is not contained in $D' := D_2 \cup \cdots \cup D_n$. The open subset $U' := U \smallsetminus D'$ is non-empty and $D_1$ is a neighbourhood of every point in $D_1$.
\end{proof}

\begin{proof}[Proof of Proposition \ref{Prop:ZariskiTrivialSpaces} (A. Ducros)] Let $x \in X$ be a point having an affinoid neighbourhood $D$. Let $\tilde{D}$ be the graded reduction of $D$ and $\tilde{x}$ be the image of $x$ in $D$ through the reduction map $\red \colon D \to \tilde{D}$. Let $\xi$ be a closed point in the closure of $\tilde{x}$ in $\tilde{D}$. By the ``Graded Nullstellensatz'' \cite[Corollaire 2.11]{PoineauAngeliques} there are real numbers $r_1, \dots, r_n$, free over $|k^\times|$, such that the residue field $\tilde{k}(\xi)$ at $\xi$ is a finite extension of the graded reduction $\tilde{k}_r$ of $k_r$.  By seeing $\tilde{k}_r$ as a graded sub-field of $\tilde{k}(\xi)$ we may consider functions $f_1, \dots, f_n \in \Oc(D)$ whose reductions evaluated at $\xi$ are equal to the variables $t_1, \dots, t_n$ of the field $\tilde{k}_r$.

On the tube $U := \red^{-1}(\xi)$, the absolute value of $f_i$ is identically equal to $r_i$. This permits to endow $U$ with a structure of $k_r$-analytic space. The $k_r$-analytic space $U$ is without boundary because $\tilde{k}(\xi)$ is finite over $\tilde{k}_r$ \cite[Proposition 3.4]{TemkinLocalII}. Moreover, since $\xi$ belongs to the closure of $\tilde{x}$, the open set $U$ contains $x$ in its closure. In particular $U$ meets the topological interior of $D'$ in $X$. 

As $D' \cap U$ is an open subset of the $k_r$-analytic space without boundary $U$, it contains a rigid point $y$ (here one uses that $k$ is non-trivially valued). 

The point $y$ is Zariski-closed in $X$: indeed it is Zariski closed in the open subsets $D' \cap U$, $X \smallsetminus \{ y\}$ of $X$, and this two open subsets cover $X$. Since the Zariski topology on $X$ is trivial, the singleton $\{ y\}$ must be (set-theoretically) the whole space.
\end{proof}

\end{document}